\documentclass[11pt]{article}
\usepackage{declare-style-article}
\usepackage{declare-maths-article}
\usepackage{declare-text-article}
\usepackage{declare-theorems-article}
\hypersetup{
	pdfauthor 
		= {Nikita Nikolaev},
	pdftitle 
		= {Abelianisation of Logarithmic sl(2)-Connections},
	pdfsubject
		= {},
	pdfcreator
		= {},
	pdfproducer
		= {},
	pdfkeywords
		= {meromorphic connections, spectral curves, spectral networks, Stokes graph, exact WKB, abelianisation, Levelt filtrations, singular differential equations, Higgs bundles, local systems},
}

\usepackage{appendix}
\usepackage[font={footnotesize}]{caption}

\fancypagestyle{frontpage}{

\cfoot{}
\lfoot{\footnotesize
Appeared: 9 February 2019. Accepted: 30 June 2021.\\
To appear in \href{https://www.springer.com/journal/29}{\textsc{Selecta Mathematica}}\\
Contact: \href{mailto:n.nikolaev@sheffield.ac.uk}{n.nikolaev@sheffield.ac.uk}}
}

\DeclareSymbolFont{bbold}{U}{bbold}{m}{n}
\DeclareSymbolFontAlphabet{\mathbbold}{bbold}

\NewDocumentCommand{\evatlong}{m o g}{%
\IfNoValueTF{#2}
	{ { #1 }{\raisebox{-2pt}{$\big\rvert_{ #3 }$}}}
	{ \left. { #1 } \right|_{ #3 }^{ #2} }
}

\newcommand{\bbDelta}{\mathbbold{\Delta}}
\newcommand{\bbid}{\mathbbold{1}}

\renewcommand{\vec}[1]{{\smash{\overset{\smash{\text{\raisebox{-0.2em}{\protect\scalebox{0.75}{$\rightarrow$}}}}}{#1}}}}

\begin{document}


\newgeometry{left=3.5cm, right=3.5cm, top=2cm, bottom=3cm}

\title{Abelianisation of Logarithmic $\frak{sl}_2$-Connections}

\author{Nikita Nikolaev}

\affil{\small School of Mathematics and Statistics, University of Sheffield, United Kingdom
\\
Section of Mathematics, University of Geneva, Switzerland
}

\date{2 July 2021}

\maketitle
\thispagestyle{frontpage}

\begin{abstract}
We prove a functorial correspondence between a category of logarithmic $\frak{sl}_2$-connections on a curve $X$ with fixed generic residues and a category of abelian logarithmic connections on an appropriate spectral double cover $\pi : \sf{\Sigma} \to X$.
The proof is by constructing a pair of inverse functors $\pi^\ab, \pi_\ab$, and the key is the construction of a certain canonical cocycle valued in the automorphisms of the direct image functor $\pi_\ast$.
\end{abstract}

{\small

\textbf{2020 MSC:} 
\href{https://zbmath.org/classification/?q=cc%3A53C05}{53C05 (primary)},
\href{https://zbmath.org/classification/?q=cc%3A34M03}{34M03},
\href{https://zbmath.org/classification/?q=cc%3A34M35}{34M35},
\href{https://zbmath.org/classification/?q=cc%3A34M40}{34M40},
\href{https://zbmath.org/classification/?q=cc%3A34M45}{34M45}
}


{\small
\tableofcontents
}

\newpage
\restoregeometry

\section{Introduction}

This paper describes an approach to analysing meromorphic connections on Riemann surfaces.
The technique, called \textit{abelianisation}, is to introduce a decorated graph $\Gamma$ on a Riemann surface $X$ in order to establish a correspondence between meromorphic connections on vector bundles of higher rank over $X$ and meromorphic connections on line bundles (which we call \textit{abelian connections}) over a multi-sheeted ramified cover $\sf{\Sigma} \to X$.
Namely, given a flat vector bundle $\cal{E}$ on $X$, an application of the standard local theory of singular differential equations near each pole allows one to extract valuable asymptotic information in the form of locally defined flat filtrations on $\cal{E}$, first discovered by Levelt \cite{MR0145108}.
These filtrations, often called \textit{Levelt filtrations}, can be organised into a single flat line bundle $\cal{L}$ but only over $\sf{\Sigma}$, and $\cal{E}$ can be recovered from $\cal{L}$ using the combinatorial data encoded in $\Gamma$.

\paragraph{Main result.}
In this paper, we restrict our attention to the simplest case of $\frak{sl}_2$-connections with logarithmic singularities and generic residues.
Our main result (\Autoref{191115100309}) is a natural equivalence between a category of $\frak{sl}_2$-connections on $X$ and a category of logarithmic abelian connections on a double cover $\sf{\Sigma}$ of $X$.
More precisely, fix $(X,D)$ a compact smooth complex curve with a finite set of marked points, fix the data of generic residues along $D$, and choose an appropriate meromorphic quadratic differential $\phi$ on $X$ with double poles along $D$.
Then $\phi$ gives rise to a double cover $\pi : \sf{\Sigma} \to X$ (called the \textit{spectral curve}) ramified at $R \subset \sf{\Sigma}$, a graph $\Gamma$ on $X$ (called the \textit{Stokes graph}), and a transversality condition on the Levelt filtrations extracted at nearby poles as dictated by $\Gamma$.
Then there is a natural equivalence of categories:
\eqn{
\begin{tikzcd}[ampersand replacement = \&]
	\begin{Bmatrix}
			\text{$\frak{sl}_2$-connections on $X$}
		\\	\text{with logarithmic poles at $D$}
		\\	\text{with very generic residues}
		\\	\text{transverse with respect to $\Gamma$}
	\end{Bmatrix}
					\ar[r, phantom, "\sim" description]
					\ar[r, shift left, "\pi^\ab_\Gamma"]
		\&
					\ar[l, shift left, "\pi_\ab^\Gamma"]
	\begin{Bmatrix}
			\text{abelian connections on $\sf{\Sigma}$}
		\\	\text{with logarithmic poles}
		\\  \text{at $\pi^{-1} (D) \cup R$ with fixed residues}
		\\	\text{equipped with odd structure}
	\end{Bmatrix}
\end{tikzcd}
\fullstop
}
Given a flat vector bundle $\cal{E}$ on $X$, the \textit{abelianisation functor} $\pi^\ab_\Gamma$ extracts Levelt filtrations along $D$ and glues them into a flat line bundle $\cal{L}$ over $\sf{\Sigma}$.
In order to recover $\cal{E}$ from $\cal{L}$, the main difficulty is that the naive guess that $\cal{E}$ is the pushforward $\pi_\ast \cal{L}$ is incorrect because $\pi_\ast \cal{L}$ necessarily has logarithmic singularities along the branch locus.
The solution is to realise the combinatorial content of the Stokes graph $\Gamma$ in cohomology: we construct a canonical cocycle $\Voros$ on $X$ (called the \textit{Voros cocycle}) which deforms the pushforward functor $\pi_\ast$, as a functor, and this deformation is the \textit{nonabelianisation functor} $\pi_\ab^\Gamma$.
The Voros cocycle is constructed in a completely standardised and combinatorial way from the Stokes graph $\Gamma$.
This is significant because it means $\Voros$ is constructed without reference to any specific choice of $\cal{E}$ or $\cal{L}$, thereby setting up an equivalence of categories.

\paragraph{Context: spectral networks and exact WKB.}
Analysis of higher rank connections using abelian connections over a multi-sheeted cover has previously appeared in the context of spectral networks \cite{MR3115984,MR3003931,MR3147409,MR3500424,MR3613514}, and even earlier from a different point of view in the context of the exact WKB analysis; e.g., \cite{MR729194,MR1209700,MR2182990}.
The purpose of our work is to give a mathematical formulation of abelianisation of connections, and this paper is the first and important step in this direction.
Our point of view, via the deformation theory of the pushforward functor, sheds light on the mathematical content of the methods of spectral networks and the exact WKB analysis, unifying the insights coming from these theories.
Indeed, the local expressions for the Voros cocycle $\Voros$ involve precisely the same type of unipotent matrices that appear in the pioneering work of Voros on the exact WKB analysis \cite{MR729194} (we call $\Voros$ the \textit{Voros cocycle} exactly for this reason).
At the same time, the off-diagonal terms of $\Voros$ are given in terms of abelian parallel transports along canonically defined paths on the spectral curve.
These appeared in the work of Gaiotto--Moore--Neitzke \cite{MR3115984} which inspired the current project.
In fact, one of the main achievements of this paper is giving a clear mathematical explanation that the path-lifting rule appearing in \cite{MR3115984} emerges simply from the repeated application of the Voros cocycle.

\paragraph{Outlook.}
Abelianisation of connections can be seen as generalising the abelianisation of Higgs bundles \cite{MR887284,MR998478} (a.k.a. the spectral correspondence, which is a key step in the analysis of Hitchin integrable systems and the geometric Langlands programme) to flat bundles.
Indeed, \Autoref{191115181734} shows that the abelianisation line bundle $\cal{L}$ is the correct analogue of the spectral line bundle.
It was also conjectured in the work of Gaiotto--Moore--Neitzke \cite{MR3115984} that such a procedure of abelianisation of connections should yield symplectic cluster coordinates on moduli spaces of meromorphic connections.
This article (which is an extension of the work the author completed in his thesis \cite{MR3809826}) is thus the first important step in realising this programme in mathematical terms.

\paragraph{Content.}
The article is dedicated to the proof of \Autoref{191115100309}, which proceeds by constructing the functors $\pi^\ab_\Gamma, \pi_\ab^\Gamma$ and showing that they form an inverse equivalence.
\Autoref{181126145627} and \Autoref{191115181734} give a summary of the main properties of the relationship between $(\cal{E}, \nabla)$ and its abelianisation $(\cal{L}, \de)$.
We also make the curious observation that the nonabelian Voros cocycle may itself be abelianised: there is an abelian cocycle $\bbDelta$ on the spectral curve $\sf{\Sigma}$ which completely determines the Voros cocycle $\Voros$ in the sense of \Autoref{181102200623}.

\paragraph*{Acknowledgements.}
The author wishes to thank Francis Bischoff, Aaron Fenyes, Alberto Garc\'ia-Raboso, Kohei Iwaki, Omar Kidwai, Andrew Neitzke, Steven Rayan, and Shinji Sasaki for helpful and enlightening discussions.
Many thanks also go to the anonymous referee for the very thoughtful and helpful comments.
The author expresses special gratitude to Marco Gualtieri, who suggested the problem and provided so much invaluable input, support, and encouragement as the author's thesis advisor.
At various stages, this work was supported by the Ontario Graduate Scholarship and by the NCCR SwissMAP of the SNSF.

\section{Logarithmic Connections and Spectral Curves}

Throughout this paper, let $X$ be a compact smooth complex curve and $D \subset X$ a finite set of marked points.
We assume that $D$ is nonempty with $|D| > \chi (X) = 2 - 2g_X$, where $g_X$ is the genus of $X$.
The Lie algebra $\frak{sl} (2, \Complex)$ is denoted by $\frak{sl}_2$.

\subsection{Logarithmic Connections and Levelt Filtrations}
\label{181118210908}

\paragraph{}
A \dfn{logarithmic $\frak{sl}_2$-connection} on $(X, D)$ is the data $(\cal{E}, \nabla, \MM)$ of a holomorphic rank-two vector bundle $\cal{E}$ on $X$, a $\underline{\Complex}_X$-linear map of sheaves
\eqntag{
	\nabla : \cal{E} \too \cal{E} \otimes \Omega^1_X (D)
}
satisfying the Leibniz rule $\nabla (fe) = e \otimes \dd{f} + f \nabla (e)$ for all $e \in \cal{E}, f \in \cal{O}_X$, and a trivialisation $\MM: \det (\cal{E}) \iso \cal{O}_X$ such that $\MM (\tr \nabla) \MM^{-1} = \dd$.
They form a category, which we denote by $\Conn^2_{\frak{sl}} (X,D)$.
We will often omit ``$\MM$'' from the notation.

\paragraph{Generic Levelt exponents and residue data.}
The residue sequence for $\Omega^1_X (D)$ implies that the restriction of $\nabla$ to $D$ is a well-defined $\cal{O}_D$-linear endomorphism $\Res \nabla \coleq \evat{\nabla}{D} \in H^0_X \big(\cal{End} (\evat{\cal{E}}{D}) \big)$, called the \dfn{residue} of $\nabla$ along $D$.
A further restriction of $\Res \nabla$ to any point $\pp \in D$ is an endomorphism of the fibre $\Res_\pp \nabla \in \End (\evat{\cal{E}}{\pp})$ whose eigenvalues $\pm \lambda_\pp \in \Complex$ are called the \dfn{Levelt exponents} of $\nabla$ at $\pp$.
The determinant map $\det : \cal{End} (\evat{\cal{E}}{D}) \to \cal{O}_D$ sends $\Res \nabla$ to a global section of $\cal{O}_D$:
\eqntag{
	a 
		\coleq - \det \Res \nabla
		= \set{a_\pp \coleq - \det (\Res_\pp \nabla) = \lambda_\pp^2 \in \Complex ~\big|~ \pp \in D}
		\in H^0_X (\cal{O}_D)
\fullstop
}

\begin{defn}[generic residue data]{181118161351}
The Levelt exponents $\pm \lambda_\pp$ at $\pp$ are \dfn{generic} if $\Re (\lambda_\pp) \neq 0$ and $\lambda_\pp \notin \tfrac{1}{2} \Integer$.
We will refer to any section $a \in H^0_X (\cal{O}_D)$ as \dfn{residue data}, and say it is \dfn{generic} if for each $\pp \in D$, the two square roots $\pm \lambda_\pp$ of $a_\pp$ define generic Levelt exponents.
\end{defn}

\paragraph{}
Thus, $a$ is generic if and only if each complex number $a_\pp$ is \textit{not} purely negative real or a quarter square $n^2/4$ for some $n \in \Integer$.
We will always order the generic Levelt exponents by their increasing real part: $- \lambda_\pp \prec \lambda_\pp$ if and only if $\Re (\lambda_\pp) > 0$.
The assumption that $\Re (\lambda_\pp) \neq 0$ is necessary for the construction in this paper because we will use the ordering $\prec$, but the assumption that $\lambda_\pp \notin \tfrac{1}{2} \Integer$ (usually called \dfn{non-resonance}) can be removed without a great deal of difficulty; in this paper, however, we restrict ourselves to this simplest situation and generalisations will appear elsewhere.

\begin{example}{210610135430}
Perhaps the most familiar explicit example is the following.
Take $X \coleq \PPP^1$, fix $d \geq 3$ distinct points $D \coleq \set{u_1, \ldots, u_{d}} \subset \PPP^1$, and $d$ constant matrices $\AA_1, \ldots, \AA_d \in \frak{sl}_2$ with $\AA_1 + \ldots + \AA_d = 0$.
We usually choose an affine coordinate $z$ on $\PPP^1 \eqcol \PPP^1_z$ such that $u_d$ is the point at infinity.
Then the trivial rank-two vector bundle $\cal{E} = \cal{O}_{\PPP^1} \oplus \cal{O}_{\PPP^1}$ is equipped with a logarithmic connection $\nabla$ defined with respect to the standard basis for $\cal{E}$ by the following formula in the affine coordinate charts $z$ and $w = z^{-1}$:
\eqntag{\label{210610202346}
	\nabla \coleq \dd + \sum_{i=1}^{d-1} \frac{\AA_i}{z-u_i} \dd{z}
\qtext{and}
	\nabla = \dd - \sum_{i=1}^{d-1} \frac{\AA_i}{1-u_i w} \frac{\dd{w}}{w}
\fullstop
}
Evidently, $\nabla$ has logarithmic singularities at each point $u_i$ with residue ${\Res_{u_i} \nabla = \AA_i}$.
The residue $\Res \nabla$ along $D$ is then simply the full collection of the chosen matrices $\set{\AA_1, \ldots, \AA_d}$.
The eigenvalues $\pm \lambda_i \in \Complex$ of each $\AA_i$ are the Levelt exponents of $\nabla$, so the residue data of $\nabla$ is $a = \set{\lambda_1^2, \ldots, \lambda_d^2}$.
\end{example}

\paragraph{}
The central object of study in this paper is the category of logarithmic $\frak{sl}_2$-connections on $(X,D)$ with fixed generic residue data $a$, for which we shall use the following shorthand notation:
\eqntag{
	\Conn_X^2 \coleq \Conn^2_{\frak{sl}} (X,D; a) \subset \Conn^2_{\frak{sl}} (X,D)
\fullstop
}

\paragraph{Local diagonal decomposition.}
Fix a point $\pp \in D$, and consider a connection germ $(\cal{E}_\pp, \nabla_\pp)$ at $\pp$ with generic Levelt exponents $\pm \lambda_\pp$ at $\pp$, where $\Re (\lambda) > 0$.
A coordinate trivialisation $\cal{E}_\pp \iso \Complex \set{z}^2$ transforms $\nabla_\pp$ to a logarithmic $\frak{sl}_2$-differential system $\dd + \AA (z) z^{-1} \dd{z}$, where $\AA (z)$ is some $\frak{sl}_2$-matrix of holomorphic function germs.
By \cite[Theorems 5.1, 5.4]{MR0460820}, there exists a holomorphic $\SL_2$ gauge transformation which transforms the given differential system into the diagonal system $\dd + \diag (- \lambda_\pp, + \lambda_\pp) z^{-1} \dd{z}$ which depends only on $\lambda_\pp$ and $z$.
This classical theorem about singular ordinary differential equations admits vast generalisations, but we do not need them here.
Together with the fixed ordering on the Levelt exponents, it induces a graded decomposition of $\cal{E}_\pp$ with respect to which $\nabla_\pp$ is diagonal.

\begin{prop}[Local diagonal decomposition]{181114180038}
Let $(\cal{E}_\pp, \nabla_\pp, \MM_\pp)$ be the germ of a logarithmic $\frak{sl}_2$-connection at $\pp \in D$ with generic Levelt exponents $\pm \lambda_\pp$.
Then there is a canonical ordered decomposition
\eqntag{
	\cal{E}_\pp \iso \Lambda_\pp^- \oplus \Lambda_\pp^+
\qtext{with}
	\nabla_\pp \simeq \de_\pp^- \oplus \de_\pp^+
\fullstop{,}
}
where $(\Lambda_\pp^\pm, \de_\pp^\pm)$ is a rank-one logarithmic connection germ at $\pp$ with residue $\pm \lambda_\pp$.
Moreover, $\MM$ induces a flat skew-symmetric isomorphism $\MM_\pp : \Lambda_\pp^- \otimes \Lambda_\pp^+ \iso \cal{O}_{X,\pp}$.
\end{prop}

Here, ``skew-symmetric'' means that $\MM_\pp$ is multiplied by $-1$ under the switching map.
The order on the Levelt exponents $- \lambda_\pp \prec + \lambda_\pp$ determines a $\nabla_\pp$-invariant filtration $\cal{E}_\pp^\bullet \coleq \big( \Lambda^-_\pp \subset \cal{E}_\pp \big)$ on the vector bundle germ $\cal{E}_\pp$, which we will refer to as the \dfn{Levelt filtration} in reference to the more general such concept studied by Levelt in his thesis \cite{MR0145108}.

We will refer to the $\nabla_\pp$-invariant filtration $\cal{E}_\pp^\bullet \coleq \big( \Lambda^-_\pp \subset \cal{E}_\pp \big)$, given by the order on the Levelt exponents $- \lambda_\pp \prec + \lambda_\pp$, as the \dfn{Levelt filtration} on the vector bundle germ $\cal{E}_\pp$.
Clearly, any pair of logarithmic $\frak{sl}_2$-connection germs $(\cal{E}_\pp, \nabla_\pp), (\cal{E}'_\pp, \nabla'_\pp)$ with the same generic Levelt exponents $\pm \lambda_\pp$ at $\pp$ are isomorphic and any such isomorphism is necessarily diagonal with respect to the diagonal decompositions.
Any morphism $(\cal{E}_\pp, \nabla_\pp) \to (\cal{E}'_\pp, \nabla'_\pp)$ necessarily preserves the Levelt filtration, as the following lemma explains.

\begin{lem}{181115123246}
Suppose $(\cal{E}, \nabla), (\cal{E}', \nabla')$ is a pair of $\frak{sl}_2$-connections on $(X,D)$ with the same local exponents along $D$.
Then any morphism $\varphi : \cal{E} \to \cal{E}'$ of connections restricts to a map $\varphi_\pp : \cal{L}_\pp \to \cal{L}'_\pp$ between the Levelt subbundles near $\pp$ for every $\pp \in D$.
\end{lem}

\newpage

\begin{proof}
Let $\psi, \psi'$ be (possibly multivalued) flat sections of $\cal{L}_\pp, \cal{L}'_\pp$, respectively, and let $\psi''$ be a (possibly multivalued) flat section of $\cal{E}'$ linearly independent from $\psi'$, all defined over a punctured neighbourhood of $\pp$.
Thus, $(\psi', \psi'')$ is an ordered basis of (multivalued) flat sections of $\cal{E}'$ over a punctured neighbourhood of $\pp$.
The sections $\psi, \psi'$ decay to $0$ as $x \to 0$ for any local coordinate $x$ centred at $\pp$, whilst $\psi''$ is unbounded.
Being a morphism of connections, the map $\varphi$ necessarily sends $\psi$ to a constant linear combination of $a\psi' + b\psi''$.
But since $\varphi$ is holomorphic at $\pp$ and $\psi$ decays at $\pp$, the section $\varphi (\psi)$ must also decay to $0$ as $x \to 0$.
This forces $b = 0$.
\end{proof}

Note: one may alternatively choose a sectorial neighbourhood of $\pp$ to dismiss the question of multivaluedness.
In this case, we need to use that fact that $\varphi$ is bounded as $x \to 0$.

\begin{example}{210610161750}
Continuing \autoref{210610135430}, assume that $\nabla$ has generic residue data, and restrict our attention to the disc germ of, say, the singularity $u_1$.
There is an $\SL_2$ matrix $\GG = \GG (z)$, holomorphic at $z = u_1$, such that
\eqntag{
	\GG \nabla \GG^{-1} = \dd + \mtx{-\lambda_1 & \\ & + \lambda_1} \frac{\dd{z}}{z - u_1}
\fullstop
}
Then the line subbundles $\Lambda_{1}^-, \Lambda_1^+$ are generated by $e_- \coleq \GG^{-1} {\tiny\mtx{1\\0}}$ and $e_+ \coleq \GG^{-1} {\tiny\mtx{0\\1}}$.
\end{example}

\subsection{Logarithmic Connections and Double Covers}

Logarithmic connections can be pulled back and pushed forward along ramified covers.
In this section we describe these operations, restricting ourselves to the simplest case of double covers $\pi : \sf{\Sigma} \to X$ with simple ramification and which are trivial over the polar divisor $D$.
Thus, let $C \coleq \pi^{-1} (D)$ and let $R \subset \sf{\Sigma}$ be the ramification divisor.
Here and everywhere, we assume that $R$ has no higher multiplicity and that the branch locus $B \coleq \pi (R) \subset X$ is disjoint from $D$.
We denote by $\sigma : \sf{\Sigma} \to \sf{\Sigma}$ the canonical involution.

\paragraph{Odd abelian connections.}
Connections on line bundles are sometimes called \dfn{abelian connections}.
The line bundle $\cal{O}_\sf{\Sigma} (R)$ carries a canonical logarithmic connection $\de_{R}$, defined to be the connection for which the canonical map $\cal{O}_\sf{\Sigma} \to \cal{O}_\sf{\Sigma} (R)$ is flat.
Explicitly, if $z$ is a local coordinate on $\sf{\Sigma}$ vanishing at $\rr \in R$, then the local section $z^{-1} \in \cal{O}_\sf{\Sigma} (R)$ gives a trivialisation, in which $\de_{R}$ is given by
\eqntag{
	\de_{R} (z^{-1}) = \dd{(z^{-1})} = -z^{-1} \dd{z} \otimes z^{-1}
\fullstop{,}
	\qqtext{i.e.,}
	\de_{R} = \dd - z^{-1} \dd{z}
\fullstop
}

\begin{defn}[odd abelian connection]{180511123511}
An \dfn{odd abelian logarithmic connection} on $(\sf{\Sigma}, R \cup C)$ is the data $(\cal{L}, \de, \mu)$ consisting of an abelian logarithmic connection on $(\sf{\Sigma}, R \cup C)$ equipped with a skew-symmetric isomorphism
$
	\mu : \cal{L} \otimes \sigma^\ast \cal{L} \iso \cal{O}_\sf{\Sigma} (R)
$
intertwining $\de \otimes \sigma^\ast \de$ and $\de_{R}$.
\end{defn}

Here, ``skew-symmetric'' means $\mu$ satisfies $\sigma^\ast \mu = - \mu$.
Abelian connections with a similar structure but over the \textit{punctured} curve $\sf{\Sigma} \setminus C \cup R$ have appeared in \cite[\S4.2]{MR3500424} under the name \textit{equivariant connections}.
We refer to the isomorphism $\mu$ as the \dfn{odd structure} on $(\cal{L}, \de)$.

Odd abelian connections form a category 
\eqntag{
	\Conn^1_{\sf{\Sigma}} = \Conn^1_\text{odd} (\sf{\Sigma}, R \cup C)
}
where morphisms are morphisms of connections $\varphi : \cal{L} \to \cal{L}'$ that intertwine the odd structures $\mu, \mu'$ in the sense that $\mu' \circ (\varphi \otimes \sigma^\ast \varphi) = \mu$.
It is easy to check that if $\mu_1, \mu_2$ are any two odd structures on the same abelian connection $(\cal{L}, \de)$, then $(\cal{L}, \de, \mu_1) \cong (\cal{L}, \de, \mu_2)$, and there are exactly two such isomorphisms.

\begin{prop}[residues of odd connections]{181116155648}
The residue of any odd abelian connection $(\cal{L}, \de, \mu)$ at a ramification point is $-1/2$.
In particular, the monodromy of $\de$ around a ramification point is $-1$.
Furthermore, if $\pp \in D$ and $\pp_\pm \in C$ are the two preimages of $\pp$, then the residues of $\de$ at $\pp_\pm$ satisfy
\eqntag{
	\Res_{\pp_-} \de + \Res_{\pp_+} \de = 0
\fullstop
}
\end{prop}

\begin{proof}
The residue of $\de_{R}$ at $\rr \in R$ is $-1$.
If $\lambda = \Res_\rr \de$, then the residue of the connection $\de \otimes \sigma^\ast \de$ at $\rr$ is $2 \lambda$, so the odd structure on $\cal{L}$ forces $\lambda = -1/2$.
Next, since $\sigma (\pp_-) = \pp_+$, the residue at $\pp_-$ of $\sigma^\ast \de$ is equal to the residue of $\de$ at $\pp_+$.
This means $\de \otimes \sigma^\ast \de$ has residue $\Res_{\pp_-} \de + \Res_{\pp_+} \de$ at $\pp_-$.
But the residue of $\de_{R}$ at $\pp_-$ is $0$, so the odd structure on $\cal{L}$ forces the identity.
\end{proof}

By using the residue theorem for connections \cite[Cor. (B.3), p. 186]{MR853449}, it is easy to compute the degree of a line bundle carrying an odd connection.

\begin{prop}{181109210007}
If $(\cal{L}, \de, \mu) \in \Conn^1_\textup{odd} (\sf{\Sigma}, R \cup C)$, then 
\eqntag{
	\deg (\cal{L}) = \tfrac{1}{2} |R| = - \deg (\pi_\ast \cal{O}_\sf{\Sigma})
\fullstop
}
\end{prop}

\paragraph{Pullback and pushforward of connections.}
The pullback of $\cal{O}_X$-modules along $\pi$ extends to a pullback functor on connections 
\eqntag{
	\pi^\ast : \Conn^2_{\frak{sl}} (X,D) \to \Conn^2_{\frak{sl}} (\sf{\Sigma}, C)
}
by the rule $\pi^\ast \nabla (\pi^\ast e) = \pi^\ast (\nabla e)$ for any local section $e \in \cal{E}$.
Clearly, the Levelt exponents of $\nabla$ at $\pp \in D$ and the Levelt exponents of $\pi^\ast \nabla$ at any preimage $\smalltilde{\pp} \in C$ of $\pp$ are the same.
More interesting is pushing connections forward along $\pi$.
The direct image functor $\pi_\ast$ of $\cal{O}_\sf{\Sigma}$-modules can be used to pushforward connections from $\sf{\Sigma}$ down to $X$, but the relationship between the polar divisors is more complicated (see \cite[proposition 2.17]{MR3808258} for more generality).

\begin{prop}[pushforward of odd abelian connections]{181116135914}
The direct image $\pi_\ast$ extends to a functor
\eqntag{\label{181126184346}
	\pi_\ast : \Conn^1_\textup{odd} (\sf{\Sigma}, R \cup C) \too \Conn^2_{\frak{sl}} (X, B \cup D)
\fullstop
}
Moreover, for any $\de \in \Conn^1_\textup{odd} (\sf{\Sigma}, R \cup C)$, 
if $\pm \lambda \in \Complex$ are its residues at the two preimages $\pp_\pm \in C$ of a point $\pp \in D$, then the Levelt exponents of $\pi_\ast \de$ at $\pp$ are $\pm \lambda$.
\end{prop}

\begin{proof}
A logarithmic connection on $(\sf{\Sigma}, R \cup C)$ is a map $\de : \cal{L} \to \cal{L} \otimes \Omega^1_{\sf{\Sigma}} (R \cup C)$, and its direct image is therefore $\pi_\ast \de : \pi_\ast \cal{L} \to \pi_\ast \big( \cal{L} \otimes \Omega^1_{\sf{\Sigma}} (R \cup C) \big)$.
We claim that there is a canonical isomorphism
\eqntag{
	\Omega^1_\sf{\Sigma} (R \cup C) \iso \pi^\ast \Omega^1_X (B \cup D)
\fullstop
}
First, $\pi^\ast \Omega^1_X (B \cup D) = \big( \pi^\ast \Omega^1_X \big) ( \pi^\ast (B \cup D))$, where $\pi^\ast (B \cup D) = 2R \cup C$ (pulled back as a divisor).
The derivative map $\dd \pi : \cal{T}_\sf{\Sigma} \to \pi^\ast \cal{T}_X$ drops rank along $R$; i.e., it is a nonvanishing section of the line bundle $\cal{T}_{\sf{\Sigma}}^\vee \otimes \pi^\ast \cal{T}_X (-R)$, thereby inducing an isomorphism $\pi^\ast \cal{T}_X \iso \cal{T}_\sf{\Sigma} (R)$.
Dualising, we get $\pi^\ast \Omega_X^1 \iso \Omega^1_\sf{\Sigma} (-R)$.
Thus, the projection formula implies $\pi_\ast \big( \cal{L} \otimes \Omega^1_{\sf{\Sigma}} ( R \cup C ) \big) \cong \pi_\ast \cal{L} \otimes \Omega^1_X (B \cup D)$.
To check that $\pi_\ast \de$ satisfies the Leibniz rule, let $e \in \pi_\ast \cal{L}$ be a local section on some open set $U \subset X$, and $f \in \cal{O}_X (U)$.
Then $\pi_\ast \de (fe) = \pi_\ast \big( \de (\pi^\ast f \cdot e ) \big)$.
Now it is clear that the Leibniz rule for $\pi_\ast \de$ follows from the Leibniz rule for $\de$.
Therefore, $(\pi_\ast \cal{L}, \pi_\ast \de)$ is a rank-two logarithmic connection on $(X, B \cup D)$.

To show that the odd structure on $\cal{L}$ induces an $\frak{sl}_2$-structure on $\pi_\ast \cal{L}$, recall that there is a canonical isomorphism $\det (\pi_\ast \cal{L}) \cong \det (\pi_\ast \cal{O}_\sf{\Sigma}) \otimes \Nm (\cal{L})$, where $\Nm (\cal{L})$ is the norm of $\cal{L}$ \cite[Cor. 3.12]{MR2967059}.
For a double cover, there is a canonical isomorphism $\pi^\ast \Nm (\cal{L}) \cong \cal{L} \otimes \sigma^\ast \cal{L}$.
Moreover, it is easy to see that $\pi^\ast \det (\pi_\ast \cal{O}_\sf{\Sigma})$ is canonically isomorphic to $\cal{O}_\sf{\Sigma} (-R)$.
The statement about the residues is obvious because $\pi$ is unramified over $D$.
\end{proof}

\paragraph{Image of $\pi_\ast$.}
One can show that the monodromy of $\pi_\ast \de$ around the branch locus $B$ is a quasi-permutation representation of the double cover $\sf{\Sigma} \to X$ \cite{MR2060367}.
As a result, no connection on $(X,D)$ is the pushforward of an abelian connection on $\sf{\Sigma}$.
In other words, the image of the pushforward functor $\pi_\ast$ in $\Conn^2_{\frak{sl}} (X, B \cup D)$ does not even intersect the subcategory $\Conn^2_{\frak{sl}} (X, D)$.
Abelianisation fixes this problem: in \autoref{181130123445}, we will explicitly construct a deformation of the pushforward functor $\pi_\ast$ which \textit{does} map into $\Conn^2_{\frak{sl}} (X, D)$.

\subsection{Spectral Curves for Quadratic Differentials}
\label{181111161446}

Let $\phi$ be a quadratic differential on $(X,D)$, by which we mean a meromorphic quadratic differential on $X$ with at most order-two poles along $D$; i.e., it is a global holomorphic section of $S^2 \Omega^1_X (2D)$.
The standard reference is \cite{MR743423}; see also \cite[\S\S2,3]{MR3349833}.
By the Riemann-Roch Theorem,
\eqntag{\label{171108125156}
	\dim H^0_X \big( S^2 \Omega^1_X (2D) \big) 
		= 2 |D| + 3 g_X - 3
\fullstop
}

\paragraph{Quadratic residue.}
In any local coordinate $x$ centred at $\pp \in D$, a quadratic differential $\phi$ with a double pole at $\pp$ is expanded as $\phi = (a_\pp x^{-2} + \cdots) \dd{x}^2$.
The coefficient $a_\pp \in \Complex$ is a coordinate-independent quantity, called the (\dfn{quadratic}) \dfn{residue} of $\phi$ at $\pp$ and denoted $\Res_\pp (\phi)$.
The residue of $\phi$ along $D$ is thus a global section $a = \Res (\phi) \in H^0_X (\cal{O}_D)$, same as what we called \textit{residue data} in \autoref{181118210908}. 
There is a \dfn{quadratic residue short exact sequence}:
\eqntag{\label{181113110208}
\begin{tikzcd}[ampersand replacement = \&]
			0
				\ar[r]
	\&		S^2 \Omega^1_X (D)
				\ar[r]
	\&		S^2 \Omega^1_X (2D)
				\ar[r, "\Res"]
	\&		\cal{O}_D
				\ar[r]
	\&		0
\fullstop
\end{tikzcd}
}

\begin{lem}{190507093910}
For any $a \in H^0_X (\cal{O}_D)$, there exists a meromorphic quadratic differential $\phi$ on $(X,D)$ with $\Res (\phi) = a$.
\end{lem}

\begin{proof}
By the Kodaira Vanishing Theorem, $H^1_X \big( S^2 \Omega^1_X (D) \big) = 0$, which implies that the residue map $\Res : H^0_X \big( S^2 \Omega^1_X (2D) \big) \to H^0_X \big( \cal{O}_D \big)$ is surjective.
This means that any residue data $a$ decorating the divisor $D$ can be lifted to a quadratic differential $\phi$.
\end{proof}

\paragraph{}
In view of \eqref{171108125156}, the only configuration $(X,D)$ for which there is a unique quadratic differential $\phi$ with specified residues is $(g_X, |D|) = (0,3)$ (i.e., $\Proj^1$ with three marked points).
In this case, the three-dimensional vector space of quadratic differentials $H^0_X \big( S^2 \Omega^1_X (2D) \big)$ can be parameterised by the residues $\alpha, \beta, \gamma$ at the three points of $D$.
Identifying $(X, D)$ with $(\Proj^1, \set{0,1,\infty})$, one can show that the unique quadratic differential with residues $\alpha, \beta, \gamma$ at the double poles $0,1, \infty$ is
\eqntag{\label{190507093741}
	\phi = \frac{\gamma z^2 - (\alpha - \beta + \gamma) z + \alpha}{z^2 (z-1)^2} \dd{z}^2
\fullstop
}

\paragraph{Generic quadratic differentials.}
We will say that a quadratic differential $\phi$ is \dfn{generic} if all zeroes are simple.
The subspace of generic quadratic differentials in $H^0_X \big( S^2 \Omega^1_X (2D) \big)$ is obviously open dense given as the complement of a hypersurface.
If $(g_X, |D|) \neq (0,3)$, then the space of quadratic differentials is at least one-dimensional; but if $(g_X, |D|) = (0,3)$, this is a condition on the residues of $\phi$.
One can use \eqref{190507093741} to calculate that the open subspace of generic quadratic differentials for $(g_X, |D|) = (0,3)$ is the complement of the quadratic hypersurface
\eqntag{\label{190501132613}
	\set{ \alpha^2 + \beta^2 + \gamma^2 - 2 \alpha \beta - 2 \alpha \gamma - 2 \beta \gamma = 0 } \subset \Complex^3_{\alpha \beta \gamma} \cong H^0_X \big( S^2 \Omega^1_X (2D) \big)
\fullstop
}

\begin{lem}{1710120743170}
Let $a \in H^0_X (\cal{O}_D)$ be generic residue data.
If $(g_X, |D|) = (0, 3)$, assume in addition that $a$ is contained in the complement of the hypersurface \eqref{190501132613}.
Then there exists a generic quadratic differential $\phi$ on $(X,D)$ such that $\Res (\phi) = a$.
\end{lem}

\begin{example}{210610162950}
Consider the following examples of meromorphic quadratic differentials on $X \coleq \PPP^1_z$:
\eqnstag{
	\phi_1 &\coleq \tfrac{1}{9} \frac{z^2 - z + 1}{z^2 (z-1)^2} \dd{z}^2
\fullstop{,}
\\
	\phi_2 &\coleq \tfrac{1}{9} \frac{z^4 + 1}{z^2 (z-1)^2 (z+1)^2} \dd{z}^2
\fullstop{,}
\\
	\phi_3 &\coleq e^{3\pi i/4} \tfrac{16}{15} \frac{4 z^4 + 15i z^2 + 4}{(z^4 + 1)^2} \dd{z}^2
\fullstop
}
The quadratic differential $\phi_1$ is of the form \eqref{190501132613} with $\alpha = \beta = \gamma = 1/9$.
They respectively have double poles along $D_1 \coleq \set{0, 1, \infty}$, $D_2 \coleq \set{0, \pm 1, \infty}$, and $D_3 \coleq \set{e^{\pm \pi i /4}, e^{\pm 3\pi i /4}}$.
Each quadratic residue of $\phi_1$ and $\phi_2$ is $1/9$; each quadratic residue of $\phi_3$ is $e^{\pi i /4}$.
The quadratic differential $\phi_1$ has two simple zeros at $e^{\pm \pi i/3}$.
The quadratic differentials $\phi_2, \phi_3$ both have four simple zeros; they are respectively $e^{\pm \pi i/4}, e^{\pm 3\pi i/4}$ and $\pm \tfrac{1}{2} e^{\pi i/4}, \pm 2 e^{3\pi i/4}$.
Consequently, all three of these quadratic differentials are generic with generic residues.
\end{example}

\paragraph{The log-cotangent bundle.}
Let $Y$ be the total space of $\Omega^1_X (D)$, sometimes called the \dfn{log-cotangent bundle}, and let $p : Y \to X$ be the projection map.
Like the usual cotangent bundle, the log-cotangent bundle $Y$ has a canonical one-form, which can be constructed as follows.
Let $\theta \in H^0 \big(Y, p^\ast \Omega^1_X (D) \big)$ be the tautological section.
Then the fibre product
\eqntag{
\begin{tikzcd}[ampersand replacement=\&]
			\cal{A}
					\ar[r]
					\ar[d]
					\ar[dr, phantom, "\ulcorner", very near start]
	\&		p^\ast \cal{T}_X (-D)
					\ar[d]
\\			\cal{T}_Y
					\ar[r,"p_\ast"']
	\&		p^\ast \cal{T}_X
\fullstop{,}
\end{tikzcd}
}
exists in the category of vector bundles, because $p : Y \to X$ is a surjective submersion.
Unravelling the definition of the fibre product, we find that $\cal{A}$ consists of all vector fields on $Y$ that are tangent to the divisor $p^\ast D \subset Y$; i.e., $\cal{A} \cong \cal{T}_Y (- \log p^\ast D)$.
Finally, dualising the surjective map $\cal{A} \to p^\ast \cal{T}_X (-D)$ yields an injective morphism $p^\ast \Omega^1_X (D) \inj \Omega^1_Y (\log p^\ast D)$.
The \dfn{canonical one-form} $\eta^\phantomindex_Y \in H^0 \big( Y, \Omega^1_Y (\log p^\ast D) \big)$ on $Y$ is then defined as the image of the tautological section $\theta$ under this map.

\begin{example}{210610162432}
Take $X = \PPP^1_z$ with $D = \set{0,1,\infty}$.
Then $\Omega^1_{\PPP^1} (D)$ has a trivialisation over the affine $z$-chart given by the logarithmic one-form $z^{-1} (z-1)^{-1} \dd{z}$.
With respect to this trivialisation, the canonical one-form $\eta_\YY$ is simply $y z^{-1} (z-1)^{-1} \dd{z}$ where $y$ is the linear coordinate in the fibre.
\end{example}

\paragraph{The spectral curve.}
If $\phi$ is a quadratic differential on $(X,D)$, then $p^\ast \phi$ is a section of $S^2 \big( \Omega^1_Y (\log p^\ast D) \big)$ via $p^\ast \Omega^1_X (D) \inj \Omega^1_Y (\log p^\ast D)$.
The \dfn{spectral curve} of $\phi$ is the zero locus in $Y$ of the section $\eta^2_Y - p^\ast \phi \in S^2 \big( \Omega^1_Y (\log p^\ast D) \big)$:
\eqntag{\label{181109194123}
	\sf{\Sigma} \coleq \sfop{Zero} \big( \eta^2_Y - p^\ast \phi \big)
\fullstop
}
We denote by $\pi : \sf{\Sigma} \to X$ the restriction to $\sf{\Sigma}$ of the canonical projection $p : Y \to X$.
We also denote the ramification divisor by $R \subset \sf{\Sigma}$ and the branch divisor by $B \subset X$.
As a double cover, $\sf{\Sigma}$ is equipped with a canonical involution $\sigma: \sf{\Sigma} \to \sf{\Sigma}$.

If $\phi$ is generic, then $\sf{\Sigma}$ is embedded in $Y$ as a smooth divisor, and the projection $\pi : \sf{\Sigma} \to X$ is a simply ramified double cover, branched exactly at the zeroes of $\phi$, and trivial over the points of $D$.
Its genus is  $g^\phantomindex_\sf{\Sigma} = |D| + 4 (g^\phantomindex_X - 1) + 1$.
(see, e.g., \cite[remark 3.2]{MR998478}).
Using the Riemann-Hurwitz formula, the number of ramification points $|R|$ of $\pi$, which is the same as the number of zeroes $|B|$ of $\phi$, is 
\eqntag{\label{181009134400}
	|R| = |B| = 2 |D| + 4 (g^\phantomindex_X - 1)
\fullstop
}

\begin{example}{210610163747}
For the quadratic differential $\phi_1$ from \autoref{210610162950}, the spectral curve $\sf{\Sigma}$ has genus $0$, hence is a copy of $\PPP^1$.
If we trivialise $\Omega^1_{\PPP^1} (D)$ over the affine $z$-chart using the differential form $z^{-1} (z-1)^{-1} \dd{z}$, then $\sf{\Sigma}$ is given by the equation $y^2 = \tfrac{1}{9} (z^2 - z + 1)$.
For both quadratic differentials $\phi_2$ and $\phi_3$, the spectral curve has genus $1$, so it is an elliptic curve over $\PPP^1$, and it is given by $y^2 = \tfrac{1}{9} (z^4 + 1)$.

Notice that, although the quadratic differential $\phi_1$ is singular at the points $0,1$ in the affine $z$-chart, its spectral curve $\sf{\Sigma}$ is perfectly well-behaved above these points (see \autoref{181119191225}).
This is a manifestation of the fact that our spectral curve $\sf{\Sigma}$ is embedded inside the total space of the logarithmic cotangent bundle rather than the usual cotangent bundle.
In contrast, constructing a spectral curve of $\phi_1$ using the same equations but in the usual cotangent bundle yields a curve which escapes from the total space above the points $0,1$ (see \autoref{181119200059}).
\end{example}

\begin{figure}
\begin{minipage}{0.45\linewidth}
\centering
\includegraphics[trim=5pt 5pt 5pt 5pt, clip, height=2.5cm]{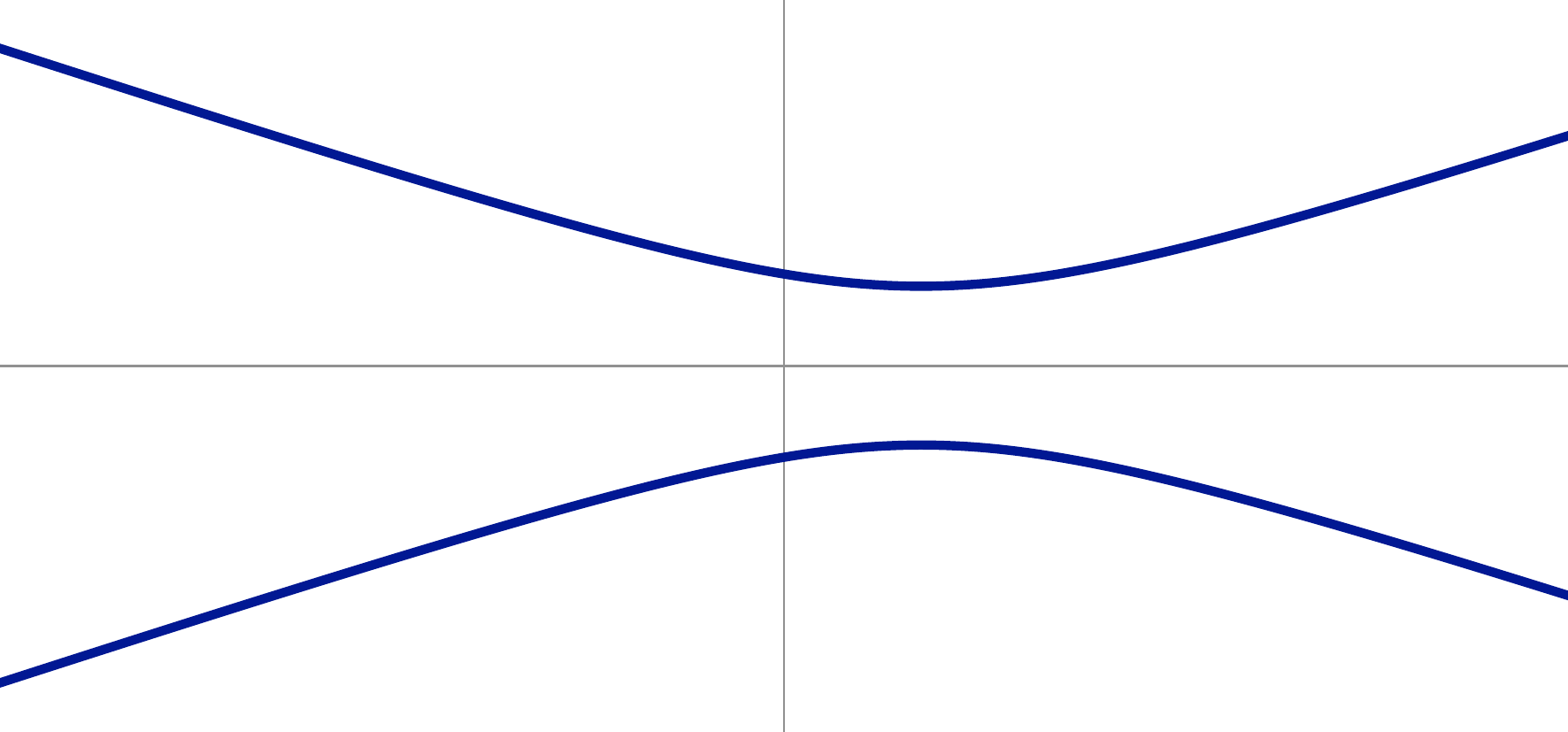}
\caption{A real slice of the total space of $\Omega^1_{\PPP^1} (D)$ over the real line in $\PPP^1_z$.
In blue is the spectral curve $\sf{\Sigma}$ of the quadratic differential $\phi_1$ from \autoref{210610162950}.}
\label{181119191225}
\end{minipage}
\hspace{0.5cm}
\begin{minipage}{0.45\linewidth}
\centering
\includegraphics[trim=5pt 5pt 5pt 5pt, clip, height=4cm]{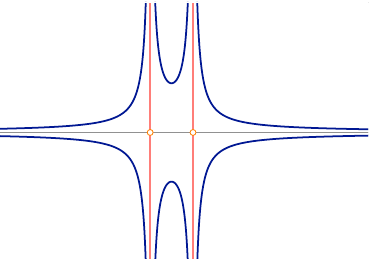}
\caption{A real slice of the total space of $\Omega^1_{\PPP^1}$ over the real line in $\PPP^1_z$.
In blue is the curve given by the equation $(y\dd{z})^2 = \phi_1$, where $\phi_1$ is the quadratic differential from \autoref{210610162950}.}
\label{181119200059}
\end{minipage}
\end{figure}

\paragraph{The canonical one-form.}
\label{181111154704}
Pulling back the canonical one-form $\eta^\phantomindex_Y$ to $\sf{\Sigma}$ yields a differential form $\eta$ with logarithmic poles along $C \coleq \pi^{-1} (D)$, called the \dfn{canonical one-form} on $\sf{\Sigma}$.
It satisfies $\eta^2 = \pi^\ast \phi$ and $\sigma^\ast \eta = - \eta$, and can therefore be thought of as the `canonical square root' of the quadratic differential $\phi$.
It has zeroes along the ramification locus $R$, and its residues at the two preimages $\pp_\pm \in C$ of any point $\pp \in D$ satisfy $\Res_{\pp_-} \eta = -  \Res_{\pp_+} \eta$ and $\big( \Res_{\pp_\pm} \eta \big)^2 = \Res_\pp \phi$.
If the residue data $a = \Res (\phi)$ is generic, we can fix an order on the preimages of $\pp$:
\eqntag{\label{181109195444}
	\pp_- \prec \pp_+
\qtext{$\coliff$}
	\Re \big( \Res_{\pp_-} \eta \big) < 0 < \Re \big( \Res_{\pp_+} \eta \big)
\fullstop
}
If $\pp_- \prec \pp_+$, we shall call $\pp_-$ a \dfn{sink pole} and $\pp_+$ a \dfn{source pole}.
The divisor $C$ is thus decomposed equally into sinks and sources $C = C^- \sqcup C^+$.

\subsection{Logarithmic Connections and Spectral Curves}
\label{191115171958}

In general, connections do not have an invariant notion of eigenvalues or eigenvectors.
However, in the presence of a spectral curve, we can make sense of these notions as follows.

\paragraph{}
Let $\pi: \sf{\Sigma} \to X$ be the spectral curve of a generic quadratic differential $\phi$ with generic residue data $a$ along $D$.
Suppose $(\cal{E}, \nabla) \in \Conn_X^2$ is a logarithmic $\frak{sl}_2$-connection on $(X,D)$ with residue data $a$.
If $\pp \in D$, let $\pm \lambda_\pp$ be the Levelt exponents at $\pp$, which by construction are the residues of $\eta$ at the preimages $\pp_\pm \in C$.
Consider the local diagonal decomposition $\cal{E}_\pp \cong \Lambda_\pp^- \oplus \Lambda_\pp^+$.

Let $z$ be a local coordinate on $\sf{\Sigma}$ centred at $\pp_\pm$ in which $\eta$ is in normal form $\pm \lambda_\pp \dd{z}/z$.
Since $\sf{\Sigma}$ is unramified over $\pp$, we also use $z$ as a local coordinate on $X$ centred at $\pp$.
If we fix a basepoint $\pp_\ast$ near $\pp$, then examining the Levelt normal form of $\nabla_\pp$ with respect to the coordinate $z$ we obtain germs of (multivalued) flat sections $\psi^\pm_\pp$ which can be expressed as $\psi^\pm_\pp = f^\pm_\pp e^\pm_\pp$, where $e^\pm_\pp$ is a (univalued) generator of $\Lambda_\pp^\pm$, and $f^\pm_\pp$ is the germ of a (multivalued) function defined in the coordinate $z$ by 
\eqntag{
	f^\pm_\pp (z) = \exp \left( - \int\nolimits_{\pp_\ast}^z \pm \lambda_\pp \dd{z'} / z' \right)
\fullstop
}
The observation is that the integrand in this expression is precisely the canonical one-form $\eta$ thought of as written in the local coordinate $z$ near $\pp$.

\paragraph{}
To express this in a coordinate-free way, let $U \subset X$ be any simply connected open neighbourhood of $\pp$ disjoint from $B$ and all other points of $D$.
Then $U$ has two disjoint preimages $U_\pm$ on $\sf{\Sigma}$ where $U_\pm$ contains $\pp_\pm$.
Let $\eta_\pm$ be the restriction of $\eta$ to $U_\pm$, and we can think of $\eta_\pm$ as being defined on $U$.
Define (multivalued) functions on the punctured neighbourhood $U^\circ \coleq U \setminus \set{\pp}$ by
\eqntag{
	f_\pm (\qq) \coleq \exp \left( - \int\nolimits_{\pp_\ast}^\qq \eta_\pm \right)
\fullstop
}
Note that the germ of $f_\pm$ at $\pp$ is precisely $f_\pp^\pm$, and that $f_\pm$ satisfies the differential equation $\dlog f_\pm = - \eta_\pm$; moreover, $f_\pm$ is nowhere-vanishing on $U^\circ$.
Analytically continue the solutions $\psi^\pm_\pp$ to multivalued flat sections $\psi_\pm$ of $\cal{E}$ over $U^\circ$, and define
\eqntag{
	e_\pm \coleq f_\pm^{-1} \psi_\pm
\fullstop
}
These sections of $\cal{E}$ form a basis of holomorphic generators over $U$ satisfying
\eqntag{
	\nabla e_\pm = \eta_\pm \otimes e_\pm
\fullstop
}
Thus, we can think of $e_\pm$ as an \dfn{eigensection} of $\nabla$ with \dfn{eigenvalue} $\eta_\pm$, and the line subbundles $\Lambda_U^\pm \subset \evat{\cal{E}}{U}$ that they generate determine the flat \textit{eigen}-decomposition of $(\cal{E}, \nabla)$ over $U$ that uniquely continues the local diagonal decomposition of $\cal{E}_\pp$:
\eqntag{
	\evat{\cal{E}}{U} \iso \Lambda_U^- \oplus \Lambda_U^+
\qqtext{with}
	\nabla \simeq \de^- \oplus \de^+
\fullstop
}

\paragraph{}
More invariantly, let $\smalltilde{U} \subset \sf{\Sigma}$ be any simply connected neighbourhood of a pole $\pp \in C = \pi^{-1} (D) \subset \sf{\Sigma}$ which is disjoint from $R$ and all other points of $C$.
Let $f$ be any (multivalued) solution of the differential equation $\dlog f = - \eta$ defined over the punctured neighbourhood $\smalltilde{U}^\ast \coleq \smalltilde{U} \setminus \set{\pp}$.
Then the same calculation as above shows that the pullback $\pi^\ast \cal{E}$ over $\smalltilde{U}$ has a section $e$ which is an eigensection of $\pi^\ast \nabla$ with eigenvalue $\eta$:
\eqntag{
	\pi^\ast \nabla e = \eta \otimes e
\fullstop
}

\subsection{The Stokes Graph}
\label{181114164952}

Fix some generic residue data $a$.
If $(g_X, |D|) = (0, 3)$, assume in addition that $a$ is contained in the complement of the hypersurface \eqref{190501132613}.
For any generic quadratic differential $\phi$ on $(X,D)$ with residues $a$, let $\sf{\Sigma}$ be its spectral curve with canonical one-form $\eta$.

\paragraph{The horizontal foliation.}
The curves $X$ and $\sf{\Sigma}$, viewed as real two-dimensional surfaces, are naturally equipped with singular foliations $\frak{F}$ and $\vec{\frak{F}}$, respectively, with the property that $\vec{\frak{F}} \tinyoverset{\pi}{\too} \frak{F}$ is the orientation double cover of $\frak{F}$.
These foliations are well-known (see, e.g., \cite{MR743423,MR523212}), and we only recall what is necessary (see \cite[\S3]{MR3349833} for a concise survey).
The foliation $\vec{\frak{F}}$ can be defined as the integration of the real distribution $\cal{ker} \big( \Im (\eta) \big)$ inside the real tangent bundle of $\sf{\Sigma}$.
Concretely, the local equation for a leaf passing through a point $\pp$ is given by $\Im \left( \: \int_{\pp}^{\zz} \eta \: \right) = 0$.
Evidently, this foliation is singular at the poles $C = \pi^{-1} (D)$ and at the ramification points $R$.
The foliation $\frak{F}$, defined as the image of $\vec{\frak{F}}$ under $\pi$, is often called the \dfn{horizontal foliation} for the quadratic differential $\phi$; it is singular at the poles $D$ and the branch points $B$.
A leaf of $\frak{F}$ (or $\vec{\frak{F}}$) is \dfn{critical} if one of its endpoints belongs to $B$ (or $R$).
A critical leaf of $\frak{F}$ is a \dfn{saddle trajectory} if both of its endpoints belong to $B$.

\paragraph{}
If the horizontal foliation $\frak{F}$ has no saddle trajectories, then by \cite[Lemma 3.1]{MR3349833} the open real surface $X \setminus (D \cup B \cup \Gamma)$, where $\Gamma$ is the union of all critical leaves of $\frak{F}$, decomposes into a finite disjoint union of topological open discs, called \textit{horizontal strips} (\autoref{181127205624}).
Similarly, the open real surface $\sf{\Sigma} \setminus (C \cup R \cup \vec{\Gamma})$, where $\vec{\Gamma}$ is the union of all critical leaves of $\vec{\frak{F}}$, is also a finite disjoint union of horizontal strips (\autoref{181127205624}).
\begin{figure}[t]
\centering
\includegraphics{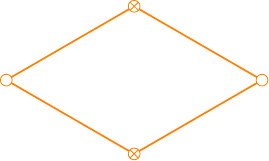}
\hspace{1cm}
\includegraphics{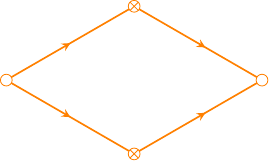}
\caption{A \dfn{horizontal strip} on $X$ (left) and on $\sf{\Sigma}$ (right).
Topologically an open disc, the boundary consists of exactly four critical leaves of $\frak{F}$ or $\vec{\frak{F}}$, two points in $D$ or $C$ (not necessarily distinct), and two points in $B$ or $R$ (necessarily distinct).
The preimage of a horizontal strip on $X$ is a pair of horizontal strips on $\sf{\Sigma}$.
\textit{Notation:} points in $B$ or $R$ are denoted by {\protect\includegraphics{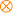}}; points in $D$ or $C$ are denoted by {\protect\includegraphics{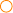}}.
}
\label{181127205624}
\end{figure}

\paragraph{Saddle-free quadratic differentials and very generic residues.}
\label{191115181047}
If the horizontal foliation $\frak{F}$ has no saddle trajectories, then the quadratic differential $\phi$ is said to be \dfn{saddle-free}.
It follows from \cite[Lemma 4.11]{MR3349833} that the subset of quadratic differentials which are saddle-free is open dense.
Note that ``saddle-free'' may be a condition on the residue data $a$.
For example, if $(g_X^\phantomindex, |D|) = (0, 3)$, the quadratic differential $\phi$ with given residues $a$ is unique (given by \eqref{190507093741}) and may fail to be saddle-free.
In this case, there are only two ramification points $\rr_\pm \in \sf{\Sigma}$, so a saddle trajectory occurs if and only if the canonical one-form $\eta$ satisfies $\Im \left( \: \int_{\rr_-}^{\rr_+} \eta \: \right) = 0$ for a path of integration in $\sf{\Sigma} \setminus C \cong \Proj^1 \setminus \set{\text{$6$ points}}$.
If $\bb_\pm \in B$ are the two branch points, then upon identifying $X \cong \Proj^1$ and choosing a branch cut in order to write $\eta$ with $\sqrt{\phi}$, where $\phi$ is given by \eqref{190507093741}, this integral can be explicitly computed in terms of logarithms and it defines a closed real-analytic subset of $\Complex^3_{\alpha \beta \gamma}$.
It therefore determines an explicit condition on the residues $a = \set{\alpha, \beta, \gamma}$ for the unique $\phi$ to be saddle-free.
We will say that residue data $a$ is \dfn{very generic} if there exists a generic saddle-free quadratic differential $\phi$ with residues $a$.

Ultimately, however, this apparent rigidity in our construction is artificial and can be removed by using a more topological argument.
We will study this as well as other non-generic situations elsewhere.

\begin{figure}[t]
\centering
\includegraphics[trim=50pt 5pt 5pt 5pt, clip, height=4cm]{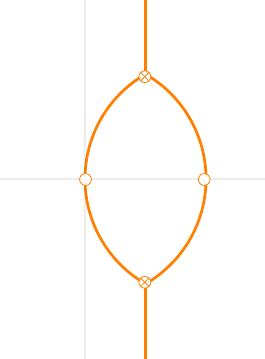}
\includegraphics[trim=5pt 5pt 5pt 5pt, clip, height=4cm]{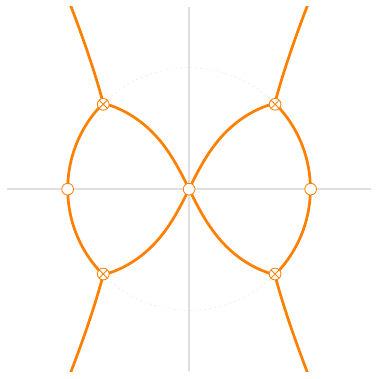}
\includegraphics[trim=5pt 5pt 5pt 5pt, clip, height=4cm]{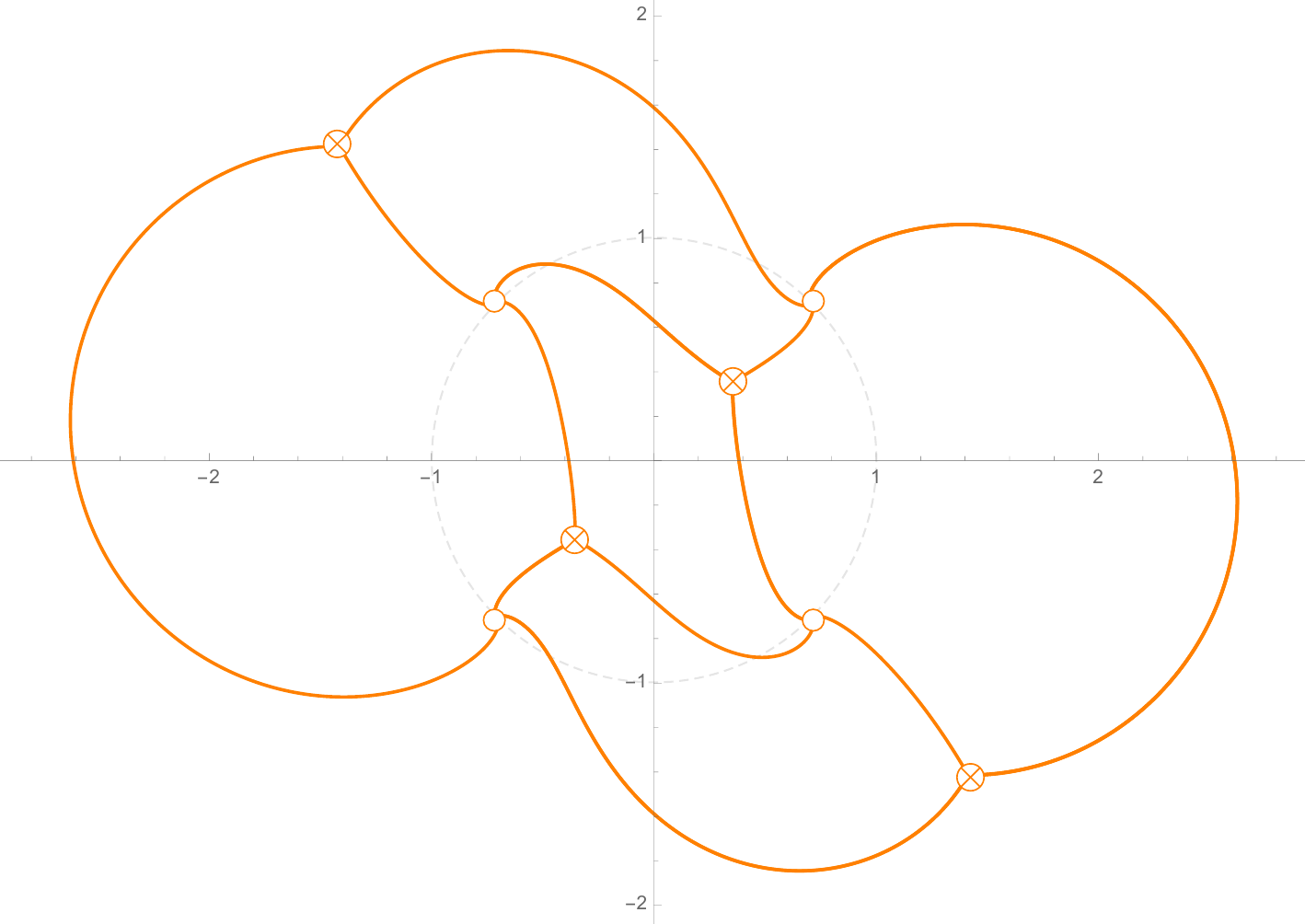}
\caption{From left to right: plot of critical trajectories of quadratic differentials $\phi_1, \phi_2, \phi_3$ from \autoref{210610162950}.
In plots 1 and 2, the trajectories that escape the picture frame tend to infinity.}
\label{210610172502}
\end{figure}

\begin{example}{210610163716}
All three quadratic differentials $\phi_1, \phi_2, \phi_3$ from \autoref{210610162950} are saddle-free.
The true plots of their critical trajectories are presented in \autoref{210610172502}.
\end{example}

\paragraph{The Stokes and spectral graphs.}
\label{200603101941}
Now we define the main combinatorial gadgets in our construction.
Let $\phi$ be a generic and saddle-free quadratic differential.

\begin{defn}[Stokes graph, spectral graph]{181031110404}
The \dfn{Stokes graph} $\Gamma$ is the graph on $X$ whose vertices are $D \cup B$ and whose edges are the critical leaves of $\frak{F}$.
The \dfn{spectral graph} $\vec{\Gamma}$ is the oriented graph on $\sf{\Sigma}$ whose vertices are $C \cup R$ and whose edges are the critical leaves of $\vec{\frak{F}}$.
\end{defn}

Thus, $\vec{\Gamma} \overset{\pi}{\to} \Gamma$ is a (ramified) orientation double cover of graphs.
Each face of $\Gamma$ and $\vec{\Gamma}$ is a horizontal strip.
We refer to the edges and the faces of $\Gamma$ as \dfn{Stokes rays} and \dfn{Stokes regions}; and to the edges and the faces of $\vec{\Gamma}$ as \dfn{spectral rays} and \dfn{spectral regions}.
The graphs $\Gamma, \vec{\Gamma}$ are bipartite with bipartitions $\Gamma_0 = D \cup B$ and $\vec{\Gamma}_0 = C \cup R$.
\begin{figure}[t]
\centering
\includegraphics{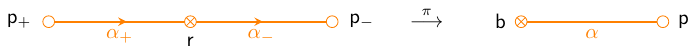}
\caption{%
Every spectral ray and every Stokes ray has a polar vertex and a ramification/branch vertex.
Depicted are the pair of opposite spectral rays $\alpha_+, \alpha_-$ on $\sf{\Sigma}$ in the preimage of the Stokes ray $\alpha$ on $X$.
\textit{Notation:}
We index Stokes rays by $\alpha, \beta, \ldots$; the corresponding positive spectral rays are denoted by $\alpha_+, \beta_+, \ldots$ and the negative ones by $\alpha_-, \beta_-, \ldots$.
}
\label{190123080855}
\end{figure}
The polar vertices $C$ are further divided into sinks and sources (cf. \autoref{181111154704}):
\begin{itemise}
\item \dfn{sink vertices $C_-$}: those where $\Re (\Res \eta) < 0$;
\item \dfn{source vertices $C_+$}: those where $\Re (\Res \eta) > 0$.
\end{itemise}
If $\pp \in D$, we will always denote its preimages in $C$ by $\pp_-, \pp_+$ where $\pp_\pm \in C_\pm$.
They satisfy the relation $\sigma (\pp_\pm) = \pp_\mp$.
All spectral rays incident to a sink/source are oriented into/out of the sink/source, so spectral rays $\vec{\Gamma}_1$ are divided by \dfn{parity}:
\begin{itemise}
\item \dfn{positive spectral rays} $\vec{\Gamma}_1^+$: polar vertex is a source;
\item \dfn{negative spectral rays} $\vec{\Gamma}_1^-$: polar vertex is a sink.
\end{itemise}

\paragraph{}
Spectral rays always occur in pairs: the involution $\sigma$ maps a spectral ray to a spectral ray of opposite parity.
Stokes rays have no natural notion of parity; instead, the preimage of every Stokes ray $\alpha \in \Gamma_1$ is a pair of opposite spectral rays $\alpha_+ \in \vec{\Gamma}_1^+, \alpha_- \in \vec{\Gamma}_1^-$ (see \autoref{190123080855}).
The graphs $\Gamma, \vec{\Gamma}$ are squaregraphs: every Stokes region is a quadrilateral with two branch vertices and two polar vertices, and its boundary is made up of four Stokes rays (\autoref{190123082314}).
\begin{figure}[t]
\centering
\includegraphics[width=\textwidth]{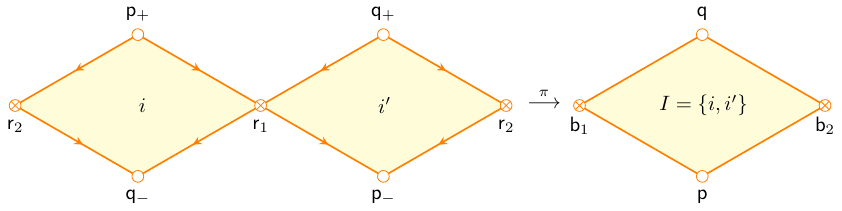}
\caption{%
Two spectral regions $i, i'$ in the preimage of the Stokes region $\II = \set{i, i'}$.
Here, $\rr_1, \rr_2 \in R$ are the ramification points above the branch points $\bb_1, \bb_2 \in B$.
\textit{Notation:}
We index faces of $\vec{\Gamma}$ by letters $i, j, k, \ldots$, though if two faces are both preimages of the same Stokes region $\II$, we will usually call them $i, i'$.
A face of $\Gamma$, whose preimage consists of faces $i, i'$ of $\vec{\Gamma}$, is indexed by the unordered pair $\II = \set{i, i'}$.
Notice that if a Stokes region $\II = \set{i,i'}$ has polar vertices $\pp, \qq \in D$, and if the spectral region $i$ has polar vertices $\pp_+, \qq_-$, then the spectral region ${i'}$ has polar vertices $\pp_-, \qq_+$.}
\label{190123082314}
\end{figure}
Similarly, every spectral region is a quadrilateral with two ramification vertices and two polar vertices (one of which is a source and one is a sink), and its boundary is made up of four spectral rays (two of which are positive and two are negative).
We index them as described in \autoref{190123082314}:
\eqntag{
	\Gamma_2
		= \set{ \II = \set{i, i'} ~\Big|~ i,i' \in \vec{\Gamma}_2 \text{ with } \sigma (i) = i'}
\fullstop
}
Each branch point has three incident Stokes rays and three incident Stokes regions, but each Stokes region has two branch vertices, so there are $3|B|$ Stokes rays and $\tfrac{3}{2}|B|$ Stokes regions in total.
So, using \eqref{181009134400},
\eqnstag{\label{181111182416}
	\big| \Gamma_1 \big| = \: 6 |D| + 12 (g^\phantomindex_X - 1)
&\qtext{and}
	\big| \Gamma_2 \big| = 3 |D| + 6 (g^\phantomindex_X - 1)
\\	\label{181111182308}
	\big| \vec{\Gamma}_1 \big| = 12 |D| + 24 (g^\phantomindex_X - 1)
&\qtext{and}
	\big| \vec{\Gamma}_2 \big| = 6 |D| + 12 (g^\phantomindex_X - 1)
\fullstop
}
Note also that $\big| \vec{\Gamma}_1^\pm \big| = \tfrac{1}{2} \big| \vec{\Gamma}_1 \big| = 6 |D| + 12 (g^\phantomindex_X - 1) = \big| \Gamma_1 \big|$.

\begin{example}{210610184429}
\autoref{210610172502} shows a plot of the Stokes graph of the quadratic differential $\phi_1$ from \autoref{210610162950}.
\autoref{210610182526} shows a more schematic rendering.
\end{example}

\begin{figure}[t]
\centering
\includegraphics[width=\textwidth]{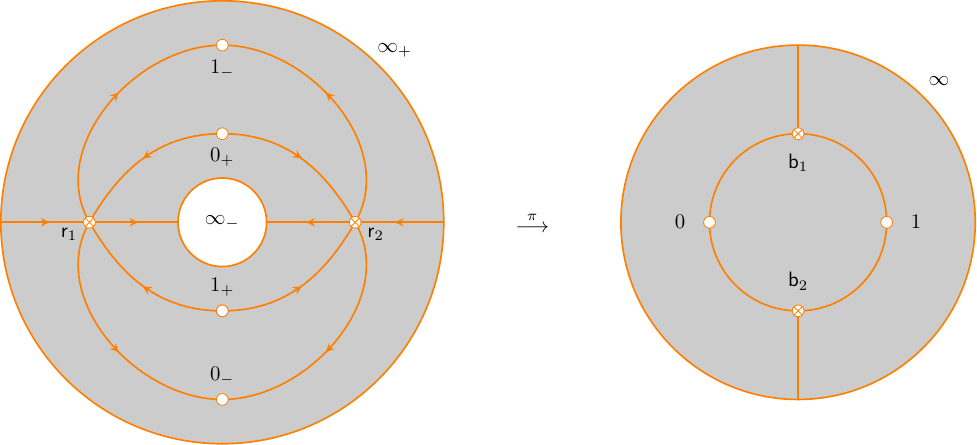}
\caption{\textit{Right:} a schematic picture of the Stokes graph $\Gamma$ (orange) of the quadratic differential $\phi_1$ from \autoref{210610162950}.
The point at infinity has been blown up to the bounding circle drawn in orange.
$\bb_1, \bb_2$ are the branch points.
\textit{Left:} the corresponding spectral graph $\vec{\Gamma}$ on the spectral curve $\sf{\Sigma} \cong \PPP^1$.
The preimages of the points $0,1,\infty$ carry a label according to whether the vertex is a sink or a source.
$\rr_1, \rr_2$ are the ramification points above $\bb_1, \bb_2$, respectively.
}
\label{210610182526}
\end{figure}

\paragraph{The Stokes open cover.}
\label{181114164806}
The graphs $\Gamma, \vec{\Gamma}$ define canonical acyclic open covers (i.e., every finite intersection is either empty or a disjoint union of contractible open sets) of the punctured curves
\eqntag{
	X^\circ \coleq X \setminus (D \cup B)
\qqtext{and}
	\sf{\Sigma}^\circ \coleq \sf{\Sigma} \setminus (C \cup R)
}
by enlarging all edges and faces as follows.
For every face $\II \in \Gamma_2$ and every edge $\alpha \in \Gamma_1$, let $U_{\II}$ and $U_{\alpha}$ be the germs of open neighbourhoods in $X^\circ$ of the face $\II$ and the edge $\alpha$, respectively.
We continue calling them \textit{Stokes regions} and \textit{Stokes rays}.
We define \textit{spectral regions} $U_i$ and \textit{spectral rays} $U_\alpha^\pm$ for all $i \in \vec{\Gamma}_2, \alpha_\pm \in \vec{\Gamma}_1$ in the same way.
We obtain what we call \dfn{Stokes open covers} of $X^\circ$ and $\sf{\Sigma}^\circ$, respectively:
\eqntag{\label{190507154747}
	\frak{U}_\Gamma \coleq \set{ U_{\II} ~\big|~ \II \in \Gamma_2}
\qqtext{and}
	\frak{U}_{\vec{\Gamma}} \coleq \set{ U_{i} ~\big|~ i \in \vec{\Gamma}_2}
\fullstop
}
If $\pp$ is a vertex of $U_\II$, then intersecting $U_\II$ with the infinitesimal disc $U_\pp$ around $\pp$ can be seen as the germ of a sectorial neighbourhood of $\pp$ (or a disjoint union of two).
In fact, the infinitesimal punctured disc $U_\pp^\ast$ centred at $\pp$ is covered by such sectorial neighbourhoods whose double intersections are the Stokes rays incident to $\pp$.

\paragraph{}
Any double intersection $U_\II \cap U_\JJ$ of Stokes regions is either a single Stokes ray or a pair of disjoint Stokes rays with the same polar vertex but necessarily different branch vertices, and there are no nonempty triple intersections.
So we define the \dfn{nerves} of these covers by
\eqntag{\label{181130172043}
	\dot{\frak{U}}_\Gamma \coleq \set{ U_{\alpha} ~\big|~ \alpha \in \Gamma_1}
\qqtext{and}
	\dot{\frak{U}}_{\vec{\Gamma}} \coleq \set{ U_\alpha^+, U_{\alpha}^- ~\big|~ \alpha \in \Gamma_1}
\fullstop
}
We adopt the following notational convention: if $U_{\alpha}$ is a Stokes ray contained in the double intersection $U_\II \cap U_\JJ$, then $U_\II, U_\JJ$ are ordered such that going from $U_\II$ to $U_\JJ$ the Stokes ray $\alpha$ is crossed anti-clockwise around the branch vertex of $U_{\alpha}$.

\paragraph{}
The restriction of the projection $\pi : \sf{\Sigma} \to X$ to any spectral region $U_i$, any spectral ray $U_\alpha^\pm$, or any infinitesimal disc $U_\pp^\pm$ around a pole $\pp_\pm$ is an isomorphism respectively onto its image Stokes region $U_\II = U_{\set{i, i'}}$, Stokes ray $U_\alpha$, or infinitesimal disc $U_\pp$ around the pole $\pp$; we denote these restrictions as follows:
\eqntag{
	\pi_i : U_i \iso U_\II
\qqtext{and}
	\pi_\alpha^\pm : U_\alpha^\pm \iso U_\alpha
\qqtext{and}
	\pi_\pp^\pm : U_\pp^\pm \iso U_\pp
\fullstop
}

\begin{example}{210610184430}
For the differential $\phi_1$ from \autoref{210610162950}, the Stokes open covers of $X^\circ = \PPP^1 \setminus \set{0,1,\infty}$ and $\sf{\Sigma}^\circ = \PPP^1 \setminus \set{ 0_\pm, 1_\pm, \infty_\pm}$ are illustrated in \autoref{210610185958}.
\end{example}

\begin{figure}
\centering
\includegraphics[width=\textwidth]{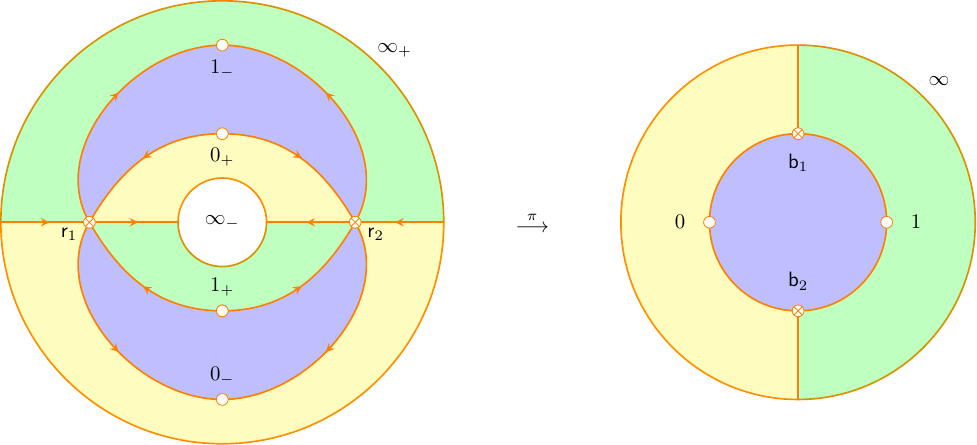}
\caption{The Stokes and spectral regions from \autoref{210610182526} are appropriately coloured to show which pair of spectral regions lie in the preimage of which Stokes region.}
\label{210610185958}
\end{figure}

\subsection{Transverse Connections}
\label{191113202418}

\paragraph{}
If $U_\II$ is a Stokes region with $\II = \set{i, i'}$, denote its polar vertices by $\pp, \pp' \in D$.
Given a connection $(\cal{E}, \nabla) \in \Conn_X^2$, consider its local diagonal decompositions $\cal{E}_\pp \cong \Lambda_\pp^- \oplus \Lambda_\pp^+$ and $\cal{E}_{\pp'} \cong \Lambda_{\pp'}^- \oplus \Lambda_{\pp'}^+$.
Let us analytically continue the flat abelian connection germs $\Lambda_\pp^-, \Lambda_{\pp'}^-$ to $U_\II$ using the flat structure on $\cal{E}$:
\eqntag{\label{190509131455}
\begin{aligned}
	(\Lambda_i, \de_i) &\coleq
	\text{ the unique continuation of $(\Lambda_\pp^-, \de_\pp^-)$ to $U_\II$}
\fullstop{,}
\\	(\Lambda_{i'}, \de_{i'}) &\coleq
	\text{ the unique continuation of $(\Lambda_{\pp'}^-, \de_{\pp'}^-)$ to $U_\II$ }
\fullstop
\end{aligned}
}

\paragraph{Transversality of Levelt filtrations.}
These continuations equip the vector bundle $\cal{E}$ over $U_\II$ with a pair of flat filtrations 
\eqntag{
	\cal{E}_{\pp,\II}^\bullet = \big( \Lambda_i \subset \cal{E}_\II \big)
\qtext{and}
\cal{E}_{\pp',\II}^\bullet = \big( \Lambda_{i'} \subset \cal{E}_\II \big)
\fullstop{,}
}
where $\cal{E}_{\pp,\II}^\bullet, \cal{E}_{\pp',\II}^\bullet$ are the unique continuations to the Stokes region $U_\II$ of the Levelt filtrations $\cal{E}_{\pp}^\bullet = \big( \Lambda_{\pp}^- \subset \cal{E}_\pp \big)$ and $\cal{E}_{\pp'}^\bullet = \big( \Lambda_{\pp'}^- \subset \cal{E}_{\pp'} \big)$, respectively.

\begin{defn}[transversality with respect to $\Gamma$]{181010152516}
We will say that a connection $(\cal{E}, \nabla) \in \Conn_X^2$ is \dfn{transverse} with respect to $\Gamma$ if for every Stokes region $U_{\II}$ the two filtrations $\cal{E}_{\pp,\II}^\bullet, \cal{E}_{\pp',\II}^\bullet$ are transverse: $\cal{E}_{\pp,\II}^\bullet \pitchfork \cal{E}_{\pp',\II}^\bullet$.
\end{defn}

In other words, the two flat line subbundles $\Lambda_i, \Lambda_{i'} \subset \cal{E}_\II$ are required to be distinct.
Such transverse connections form a full subcategory $\Conn_X^2 (\Gamma) \subset \Conn_X^2$.

That such connections exist is not difficult to see once phrased in terms of filtered local systems on $X \setminus D$.
Explicitly, for each Stokes region $U_\II$, choose a point $\xx_\II \in U_\II$ and two loops based at $\xx_\II$, one around each of the two poles {\protect\includegraphics{181117115838}} on the boundary of $U_\II$.
Since $U_\II$ is simply connected, these loops are to be chosen contractible if the corresponding pole is filled in.
Choose a `master' basepoint $\xx_0$ in $X \setminus D$, and for each $\II$, choose a path in $X \setminus D$ that connects $\xx_\II$ to $\xx_0$.
Finally, choose generators for $\pi_1 (X, \xx_0)$ represented by loops that avoid $D$.
Assigning a matrix in $\SL (2, \Complex)$ to each loop, subject to the obvious homotopy conditions, defines a local system on $X \setminus D$ which can always be extended to a logarithmic connection on $(X,D)$.
By appropriately fixing the matrices corresponding to the poles {\protect\includegraphics{181117115838}}, we can ensure that the resulting logarithmic connection has the desired residue data.
Finally, $\Gamma$-transversality cuts out a complement finitely-many algebraic conditions, one for each pair of monodromy matrices assigned to the pair of loops corresponding to each basepoint $\xx_\II$.

In fact, the same argument shows that (with respect to an appropriate topology) the subset of $\Gamma$-transverse connections is open dense.
We do not need these details here, and only mention that these and other moduli-theoretic considerations will be described in great detail in a future publication.

\begin{prop}[Semilocal diagonal decomposition of transverse connections]{181109180145}
If $(\cal{E}, \nabla, \MM) \in \Conn_X^2 (\Gamma)$, then the restriction $\cal{E}_\II \coleq \evat{\cal{E}}{U_{\II}}$ to any Stokes region $U_\II$ has a canonical flat decomposition
\eqntag{
	\cal{E}_\II \iso \Lambda_i \oplus \Lambda_{i'}
\qqtext{with}
	\nabla \simeq \de_i \oplus \de_{i'}
\fullstop{,}
}
where $(\Lambda_i, \de_i)$ and $(\Lambda_{i'}, \de_{i'})$ are defined by \eqref{190509131455}.
Moreover, the $\frak{sl}_2$-structure $\MM$ defines a flat skew-symmetric isomorphism $\MM_\II : \Lambda_i \otimes \Lambda_{i'} \iso \cal{O}_{U_\II}$.
\end{prop}

The main construction in this paper (\Autoref{191115100309}) is an equivalence between $\Conn_X^2 (\Gamma)$ and a certain category of odd abelian connections on the spectral curve $\sf{\Sigma}$.

\begin{figure}
\begin{minipage}[c]{0.47\textwidth}
\centering
\includegraphics{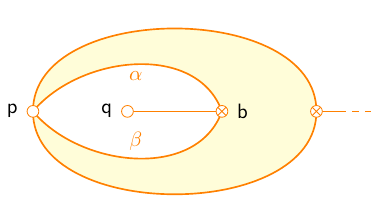}
\end{minipage}\hfill
\begin{minipage}[c]{0.5\textwidth}
\caption{A Stokes region $U_\II$ whose polar vertices coincide.
The subset of $X$ bounded by the Stokes rays $\alpha, \beta$ in the complement of $U_\II$ must contain another point $\qq \in D$, for otherwise all Stokes rays incident to the branch point $\bb$ are also incident to $\pp$.
But then the complement of $\Gamma$ has a connected component which is not a horizontal strip contradicting \cite[Lemma 3.1]{MR3349833}.
Generically, the monodromy of $\nabla$ around the pole $\qq$ does not preserve the Levelt filtration coming from $\pp$.
}
\label{181128142416}
\end{minipage}
\end{figure}

\paragraph{Transversality over Stokes rays.}
Suppose $U_\alpha$ is a Stokes ray contained in the double intersection $U_\II \cap U_\JJ$ of two adjacent Stokes regions.
Then $\cal{E}$ has two diagonal decompositions over $U_{\alpha}$:
\eqntag{\label{190514173401}
	\cal{E}_\II \iso \Lambda_i \oplus \Lambda_{i'}
\fullstop{,}
\qqquad
	\cal{E}_\JJ \iso \Lambda_j \oplus \Lambda_{j'}
\fullstop
}
Let $\pp' \in D$ be the common polar vertex of $U_\II, U_\JJ$.
Then $\Lambda_{i'}, \Lambda_{j'}$ are continuations of the same line bundle germ $\Lambda_{\pp'}^- \subset \cal{E}_\pp$, so $\Lambda_{i'} = \Lambda_{j'}$ over the Stokes ray $U_{\alpha}$.
With respect to this pair of decompositions, the identity map on $\cal{E}$ has the following upper-triangular expression, which will be exploited throughout our construction in this paper:
\eqntag{\label{190514172603}
	\evat{\Big( \cal{E}_\II \overset{\id}{=\joinrel=} \cal{E}_\JJ \Big)}{U_\alpha}
		=
	\mtx{1 & \Delta_\alpha \\ 0 & g_\alpha}
		:
	\begin{tikzcd}[ampersand replacement=\&, row sep = tiny, baseline=-2.5pt]
			\Lambda_{i'}
				\ar[d, "\oplus" description]
				\ar[r, equal, shorten >=-2.5pt, shorten <=-5pt, "1"]
	\&		\Lambda_{j'}
				\ar[d, "\oplus" description]
\\			\Lambda_i
				\ar[r, shorten >=-2.5pt, shorten <=-5pt, "g_\alpha"']
				\ar[ur, shorten >=-2.5pt, shorten <=-5pt, "\Delta_\alpha" description]
	\&		\Lambda_j
\end{tikzcd}
\fullstop
}

\begin{rem}{190509134710}
Note that in the definition of transversality with respect to $\Gamma$, it is not required that the two polar vertices $\pp, \pp'$ of $U_\II$ be different.
If $\pp = \pp'$ it may seem that no connection $\nabla$ can be transverse with respect to $\Gamma$ for such a Stokes graph, but this is not the case.
This is because the Stokes region $U_\II$ defines two disjoint sectorial neighbourhoods of $\pp$, so the two analytic continuations $\Lambda_i, \Lambda_{i'} \subset \cal{E}_\II$ of the same germ $\Lambda_\pp^-$ are generically not the same, as explained in \autoref{181128142416}.
\end{rem}

\section{Abelianisation}

\paragraph{}
As before, let $(X,D)$ be a smooth compact curve equipped with a nonempty set of marked points $D$ such that $|D| > 2 - 2g_X$.
Suppose $D$ is decorated with very generic residue data $a$ in the sense of \autoref{181118161351} and \autoref{191115181047}.
We are studying the category of logarithmic $\frak{sl}_2$-connections on $(X,D)$ with residue data $a$:
\eqntag{
	\Conn_X^2 = \Conn^2_{\frak{sl}} (X,D; a)
\fullstop
}

Our method is to choose a generic saddle-free quadratic differential $\phi$ on $(X,D)$ with residues $a$.
Let $\pi : \sf{\Sigma} \to X$ be the spectral curve of $\phi$, and let $\Gamma$ be the corresponding Stokes graph on $X$.
Consider the subcategory of connections that are transverse with respect to $\Gamma$ in the sense of \autoref{181010152516}:
\eqntag{
	\Conn_X^2 (\Gamma) \subset \Conn_X^2
\fullstop
}

\paragraph{}
The main result of this paper is that $\Conn_X^2 (\Gamma)$ is equivalent to a category of odd abelian connections on the spectral curve $\sf{\Sigma}$ as follows.
For every $\pp \in D$, let $\pm \lambda_\pp \in \Complex$ be the Levelt exponents of the residue data $a$ at $\pp$ (arranged such that $\Re (\lambda_\pp) > 0$).
Put $C \coleq \pi^{-1} (D)$, let $C_\pm$ be as in \autoref{200603101941}, let $R \subset \sf{\Sigma}$ be the ramification divisor of $\pi$, and define abelian residue data along $C \cup R$ as follows:
\eqntag{\label{181123200246}
	\underline{\lambda}
		\coleq \set{ \pm \lambda_\pp ~\big|~ \pp_\pm \in C_\pm }
			\cup \set{ -1/2 ~\big|~ \rr \in R}
\fullstop
}
Consider the category of odd abelian logarithmic connections on $(\sf{\Sigma}, C \cup R)$ with residues $\underline{\lambda}$, for which we use the following shorthand notation:
\eqntag{
	\Conn_\sf{\Sigma}^1 \coleq \Conn^1_\textup{odd} (\sf{\Sigma}, C \cup R; \underline{\lambda})
\fullstop
}

\begin{thm}[\bfseries abelianisation of logarithmic {$\frak{sl}_2$}-connections]{191115100309}
\leavevmode\\
There is a natural equivalence of categories $\Conn_X^2 (\Gamma) \cong \Conn_\sf{\Sigma}^1$. 
\end{thm}

Expressed more explicitly, this equivalence is
\eqn{
\begin{tikzcd}[ampersand replacement = \&]
	\begin{Bmatrix}
			\text{$(\cal{E}, \nabla, \MM)$}
		\\	\text{logarithmic $\Gamma$-transverse}
		\\	\text{$\frak{sl}_2$-connections on $(X,D)$}
		\\	\text{with generic Levelt exponents}
		\\  \text{$\set{\pm \lambda_\pp ~\big|~ \pp \in D}$}
	\end{Bmatrix}
					\ar[r, phantom, "\cong" description]
		\&		
	\begin{Bmatrix}
			\text{$(\cal{L}, \de, \mu)$}
		\\	\text{odd logarithmic abelian}
		\\	\text{connections on $(\sf{\Sigma}, C \cup R)$}
		\\	\text{with residues $\begin{cases} -\tfrac{1}{2} \text{ along $R$} \\ \pm \lambda_\pp \text{ at $\pp_\pm \in C^\pm$}\end{cases}$}
	\end{Bmatrix}
\end{tikzcd}
\fullstop
}

We will prove this theorem by constructing a pair functors, 
\eqntag{
\begin{tikzcd}[ampersand replacement = \&]
				\Conn_X^2 (\Gamma)
					\ar[r, shift left, "\pi^\ab_\Gamma"]
		\&		\Conn_\sf{\Sigma}^1
					\ar[l, shift left, "\pi_\ab^\Gamma"]
\end{tikzcd}
\fullstop
}
called \textit{abelianisation} and \textit{nonabelianisation} with respect to $\Gamma$; they are constructed in \autoref{181127212951} and \autoref{181130123445}, respectively.
In \Autoref{180508230111}, we prove that they form an equivalence of categories.

\subsection{The Abelianisation Functor}
\label{181127212951}

In this subsection, given an $\frak{sl}_2$-connection $(\cal{E}, \nabla, \MM) \in \Conn_X^2 (\Gamma)$, we construct an abelian connection $(\cal{L}, \de, \mu) \in \Conn_\sf{\Sigma}^1$, and show that this construction is functorial.
The idea is to extract the diagonal decompositions of $\cal{E}$ at the poles of $\nabla$, analytically continue them to the spectral regions on the spectral curve, and then glue them into a flat line bundle using canonical isomorphisms that arise due to transversality.

\paragraph{Definition at the poles.}
Given $\pp \in D$, consider the local diagonal decomposition $\cal{E}_\pp \iso \Lambda_\pp^- \oplus \Lambda_\pp^+$ from \Autoref{181114180038}.
We define $(\cal{L}, \de)$ over the infinitesimal disc $U_\pp^\pm$ around $\pp_\pm$ to be the pullback of the connection germ $(\Lambda_\pp^\pm, \de_\pp^\pm)$:
\eqntag{\label{190525201542}
	(\cal{L}_{\pp}^\pm, \de_{\pp}^\pm)
		\coleq (\pi_\pp^\pm)^\ast \big( \Lambda_\pp^\pm, \de_\pp^\pm \big)
\fullstop
}
Thus, $(\cal{L}_{\pp}^\pm, \de_{\pp}^\pm)$ is the germ of a logarithmic abelian connection at $\pp_\pm$ with residue $\pm \lambda_\pp$.
It also follows that $(\pi_\pp^\mp)^\ast \Lambda^\pm_\pp = \sigma^\ast \cal{L}_\pp^\pm$, so the pullback of the flat skew-symmetric isomorphism $\MM_\pp : \Lambda_\pp^- \otimes \Lambda_\pp^+ \iso \cal{O}_{X, \pp}$ to the disc $U_\pp^\pm$ defines a flat skew-symmetric isomorphism
\eqntag{\label{190525201538}
	\mu_\pp^\pm \coleq (\pi_\pp^\pm)^\ast \MM_\pp :
	\cal{L}_\pp^\pm \oplus \sigma^\ast \cal{L}_\pp^\mp \iso \cal{O}_{U_\pp^\pm}
\fullstop
}

\paragraph{Definition on spectral regions.}
Let $U_i \subset \sf{\Sigma}$ be a spectral region, and let $\pp_-$ be its sink vertex.
We define $(\cal{L}, \de)$ by uniquely continuing the germ $\cal{L}_\pp^-$ using the flat structure on $\pi^\ast \cal{E}$:
\eqntag{
	(\cal{L}_i, \de_i) \coleq
	\text{ the unique continuation of $(\cal{L}_\pp^-, \de_\pp^-)$ to $U_i$}
\fullstop
}
Evidently, $(\cal{L}_i, \de_i) \coleq \pi_i^\ast (\Lambda_i, \de_i)$ for $\Lambda_i$ defined by \eqref{190509131455}.
Furthermore, if $U_{i'} = \sigma (U_i)$, then $\pi^\ast_i \Lambda_{i'} = \sigma^\ast \cal{L}_{i'}$ for $\Lambda_{i'}$ defined by \eqref{190509131455}.
So if $\II = \set{i,i'}$, the pullback to $U_i$ of the $\frak{sl}_2$-structure $\MM_\II$ from \Autoref{181109180145} defines a flat skew-symmetric isomorphism
\eqntag{\label{190525201957}
	\mu_i \coleq \pi^\ast_i \MM_\II : \cal{L}_i \otimes \sigma^\ast \cal{L}_{i'} \iso \cal{O}_{U_i}
\fullstop
}

\paragraph{Gluing over spectral rays.}
For every $\alpha \in \Gamma_1$, consider the pair of opposite spectral rays $\alpha_\pm \in \vec{\Gamma}_1^\pm$, and let $\pp_\pm \in C^\pm$ be their respective polar vertices.
Let $U_\II = U_{\set{i,i'}}, U_\JJ = U_{\set{j,j'}} \subset X$ be the pair of adjacent Stokes regions which intersect along the Stokes ray $U_\alpha$ as described in \autoref{181112160949}.
\begin{figure}[t]
\begin{minipage}[c]{0.6\textwidth}
\centering
\includegraphics{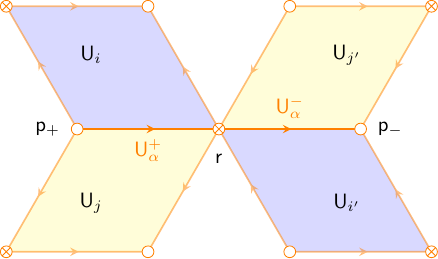}
\end{minipage}\hfill
\begin{minipage}[c]{0.37\textwidth}
\caption{%
$U_\pp^\pm$ is a pair of opposite spectral rays,
$\rr$ is their common ramification vertex, and $\pp_\pm$ are their respective polar vertices.
$U_i, U_j$ are a pair of oriented Stokes regions which have $U_\alpha^+$ in their intersection, arranged such that the ordered pair $(U_i, U_j)$ respects the cyclic anti-clockwise order around $\rr$.
Let $U_{i'} \coleq \sigma(U_i), U_{j'} \coleq \sigma(U_j)$, so $U_{\alpha}^-$ is a connected component of $U_{i'} \cap U_{j'}$.}
\label{181112160949}
\end{minipage}
\end{figure}
By transversality with respect to $\Gamma$, the vector bundle $\cal{E}$ has two diagonal decompositions over the Stokes ray $U_\alpha$:
\eqntag{\label{190514173402}
	\cal{E}_\II \iso \Lambda_i \oplus \Lambda_{i'}
\fullstop{,}
\qqquad
	\cal{E}_\JJ \iso \Lambda_j \oplus \Lambda_{j'}
\fullstop
}
Then $\Lambda_{i'}, \Lambda_{j'}$ are continuations of the same line bundle germ $\Lambda_{\pp'}^- \subset \cal{E}_\pp$, so $\Lambda_{i'} = \Lambda_{j'}$ over the Stokes ray $U_{\alpha}$.
The identity map on $\cal{E}$, written with respect to this pair of decompositions, is the upper triangular matrix \eqref{190514172603}.
We therefore define
\eqntag{\label{190523180730}
\begin{aligned}
	\Big( g_\alpha^- : \cal{L}_{i'} \iso \cal{L}_{j'} \Big)
	&\coleq (\pi_\alpha^-)^\ast \Big( 1 : \Lambda_{i'} =\joinrel= \Lambda_{j'} \Big)
\fullstop{,}
\\
	\Big( g_\alpha^+ : \cal{L}_{i} \iso \cal{L}_{j} \Big)
	&\coleq (\pi_\alpha^+)^\ast \Big( g_\alpha : \Lambda_i \iso \Lambda_j \Big)
\fullstop
\end{aligned}
}
The upper-triangular form \eqref{190514172603} of the identity map on $\cal{E}$ also implies that the gluing maps $g_\alpha^-, g_\alpha^+$ intertwine the pullbacks $\mu_i, \mu_j$ and $\mu_{i'}, \mu_{j'}$, respectively.

\paragraph{Gluing near the poles.}
For every $\pp \in D$, let $U_\pp^\pm \subset \sf{\Sigma}$ be the infinitesimal disc neighbourhoods of $\pp_\pm$.
Consider a Stokes region $U_\II = U_{\set{i,i'}}$ such that $U_i$ is incident to $\pp_+$ and $U_{i'}$ is incident to $\pp_-$.
First, the intersection of $U_{i'}$ with $U_\pp^-$ is a sectorial neighbourhood of $\pp_-$, and the line bundle $\cal{L}_{i'}$ is the unique continuation of the germ $\cal{L}_\pp^-$, so identity is the gluing map here.
On the other hand, the intersection of $U_{i}$ with $U_\pp^+$ is a sectorial neighbourhood of $\pp_+$, over which by \Autoref{181114180038} and \Autoref{181109180145} we have two decompositions $\cal{E}_\pp \iso \Lambda_\pp^- \oplus \Lambda_\pp^+$ and $\cal{E}_\II \iso \Lambda_{i} \oplus \Lambda_{i'}$.
Then the obvious isomorphism $\cal{E}_\pp \iso \cal{E}_\II$ over this double intersection implies
\eqntag{\label{190510112117}
	\Lambda_\pp^+
		\iso \big( \Lambda_\pp^- \oplus \Lambda_\pp^+ \big) \big/ \Lambda_\pp^-
		\iso \cal{E}_\pp \big/ \Lambda_\pp^-
		\iso \cal{E}_\II \big/ \Lambda_{i'}
		\iso \big( \Lambda_{i} \oplus \Lambda_{i'} \big) \big/ \Lambda_{i'}
		\iso \Lambda_i
\fullstop
}
The pullback of this map is the desired gluing map $\cal{L}_{\pp}^+ \iso \cal{L}_i$.
These gluing maps satisfy the cocycle condition, because if $U_i$ and $U_j$ are two adjacent spectral regions incident to $\pp_+$, then the identity map $\cal{E}_\II \iso \cal{E}_\JJ$ over the intersection of Stokes regions $U_\II = U_{\set{i,i'}}$ and $U_{\JJ} = U_{\set{j,j'}}$ has the upper-triangular form \eqref{190514172603}, and therefore induces an isomorphism $\cal{E}_\II \big/ \Lambda_{i'} \iso \cal{E}_\JJ \big/ \Lambda_{j'}$.
The isomorphism $\Lambda_i \oplus \Lambda_{i'} \iso \Lambda_\pp^- \oplus \Lambda_\pp^+$ given by the identity on $\cal{E}$ is unipotent.
So its determinant $\Lambda_i \otimes \Lambda_{i'} \iso \Lambda_\pp^- \otimes \Lambda_\pp^+$ intertwines $\MM_\II$ and $\MM_\pp$, and therefore also $\mu_i$ and $\mu_\pp^+$ as well as $\mu_{i'}$ and $\mu_\pp^-$.

\paragraph{Extension over ramification.}
This completes the construction of $(\cal{L}, \de, \mu)$ on the spectral curve $\sf{\Sigma}$ away from the ramification divisor $R$.
Deligne's construction \cite[pp.91-96]{MR0417174} gives an extension over $R$ with logarithmic poles and residues $-1/2$, and it is easy to check that for any such extension, $\mu$ extends uniquely to an odd structure.
Deligne extensions are unique only up to a unique isomorphism (see also \cite[Theorem IV.4.4]{Borel1987}), but it is possible to fix this ambiguity as follows (details are not important for us here and will appear elsewhere).
If $\rr \in R$ is any ramification point and $\bb = \pi (\rr)$ is the corresponding branch point, let $U_\II, U_\JJ, U_\KK$ be the three Stokes regions incident to $\bb$.
Then the germ $\cal{L}_r$ of $\cal{L}$ at $\rr$ is the pullback of the line bundle germ $\Lambda_\bb$ at $\bb$ which is defined as the kernel of the canonical map $\Lambda_\II \oplus \Lambda_\JJ \oplus \Lambda_\KK \too \cal{E}_\bb$.
As a result, we obtain an abelian connection $(\cal{L}, \de, \mu) \in \Conn_\sf{\Sigma}^1$.

\paragraph{}
Finally, functoriality of our construction readily follows from the fact that morphisms of connections necessarily preserve diagonal decompositions.

\begin{prop}{181109132636}
The assignment $(\cal{E}, \nabla, \MM) \mapsto (\cal{L}, \de, \mu)$ extends to a functor
\eqntag{
	\pi^\ab_\Gamma : \Conn_X^2 (\Gamma) \too \Conn_\sf{\Sigma}^1
\fullstop
}
\end{prop}

We call $\pi^\ab_\Gamma$ the \dfn{abelianisation functor}, and the image $(\cal{L}, \de, \mu)$ of $(\cal{E}, \nabla, \MM)$ under $\pi^\ab_\Gamma$ the \dfn{abelianisation} of $(\cal{E}, \nabla, \MM)$ with respect to $\Gamma$.
The following proposition summarises some properties of abelianisation all of which are immediate consequences of the construction.

\begin{prop}[\textbf{properties of abelianisation}]{181126145627}
Let $(\cal{E}, \nabla, \MM) \in \Conn_X^2 (\Gamma)$, and let $(\cal{L}, \de, \mu) \in \Conn_\sf{\Sigma}^1$ be its abelianisation.
\begin{enumerate}
\item \label{181127214531}
The degree of $\cal{L}$ is 
\eqn{
	\deg (\cal{L}) = - \tfrac{1}{2} |R| = - \deg (\pi_\ast \cal{O}_\sf{\Sigma})
\fullstop
}
\item \label{181127214549}
For any $\pp \in D$, let $U_\pp$ be the infinitesimal disc around $\pp$.
Let $\cal{E}_\pp \iso \Lambda_\pp^- \oplus \Lambda_\pp^+$ be the local diagonal decomposition (\Autoref{181114180038}).
Then there is a canonical flat isomorphism
\eqn{
	\pi_\ast \cal{L}_\pp = \evat{\pi_\ast \cal{L}}{U_\pp} \iso \Lambda_\pp^- \oplus \Lambda_\pp^+ \iso \cal{E}_\pp
\fullstop
}
\item \label{191115160512}
Let $U_\pp^\pm$ be the infinitesimal disc around the preimage $\pp_\pm \in C$ of $\pp$.
Recall the notation $\pi_\pp^\pm \coleq \evat{\pi}{U_\pp^\pm} : U_\pp^\pm \iso U_\pp$.
Then there are canonical flat isomorphisms
\eqn{
	\cal{L}_{\pp_\pm} = \evat{\cal{L}}{U_\pp^\pm} \iso (\pi^\pm_\pp)^\ast \Lambda_\pp^\pm \eqcol \cal{L}_\pp^\pm
\fullstop
}
\end{enumerate}

\begin{enumerate}
\setcounter{enumi}{3}
\item \label{181127214543}
Let $U_\II \subset X$ be a Stokes region with polar vertices $\pp, \pp' \in D$, and let $\cal{E}_\II \iso \Lambda_i \oplus \Lambda_{i'}$ be the semilocal diagonal decomposition of $\cal{E}$ over $U_\II$ (\Autoref{181109180145}), where $\Lambda_i, \Lambda_{i'}$ are as in \eqref{190509131455}.
Then there is a canonical flat isomorphism
\eqn{
	\evat{\pi_\ast \cal{L}}{U_\II} \iso \Lambda_i \oplus \Lambda_{i'} \iso \cal{E}_\II
\fullstop
}
\item \label{181127214735}
Let $U_i, U_{i'}$ be the spectral regions above $U_{\II}$ incident to $\pp_-, \pp'_- \in C$, respectively, and recall the notation $\pi_i \coleq \evat{\pi}{U_i} : U_i \iso U_\II$.
Then there are canonical flat isomorphisms
\eqn{
	\cal{L}_i = \evat{\cal{L}}{U_i} \iso \pi^\ast_i \Lambda_i
\qqtext{and}
	\cal{L}_{i'} = \evat{\cal{L}}{U_{i'}} \iso \pi^\ast_{i'} \Lambda_{i'}
\fullstop
}
\item \label{191115171247}
Finally, recall that $\eta$ is the canonical one-form on the spectral curve $\sf{\Sigma}$.
The abelian connection $\de - \eta$ on the abelianisation line bundle $\cal{L}$ is holomorphic along $C$; it has logarithmic poles only along the ramification divisor $R$ where it has residues $-1/2$.
\end{enumerate}
\end{prop}

The following proposition, which readily follows from the discussion in \Autoref{191115171958}, expresses the sense in which the abelianisation of connections is the analogue of abelianisation of Higgs bundles.

\begin{prop}[\textbf{spectral properties of abelianisation}]{191115181734}
For any simply connected open subset $U \subset \sf{\Sigma} \setminus R$, the abelianisation line bundle $\cal{L}$ has a generator $e$ which is an eigensection for $\de$ with eigenvalue $\eta$ (in the sense of \autoref{191115171958}); i.e., it satisfies the following equation:
\eqntag{
	\de e = \eta \otimes e
\fullstop
}
Moreover, over any spectral region $U_i \subset \sf{\Sigma}$, there is a canonical flat inclusion $\cal{L} \inj \pi^\ast \cal{E}$ with respect to which this section $e$ is an eigensection for $\pi^\ast \nabla$ with eigenvalue $\eta$:
\eqntag{
	\pi^\ast \nabla e = \eta \otimes e
\fullstop
}
\end{prop}

\begin{example}{210610200954}
Let us illustrate the above construction in the simplest possible explicit example.
Consider a logarithmic $\frak{sl}_2$-connection $(\cal{E}, \nabla)$ from \autoref{210610135430} with $d = 3$.
Namely, $X = \PPP^1$, $\cal{E} = \cal{O}_{\PPP^1}^{\oplus 2}$, and $D \coleq \set{ 0, 1, \infty}$.
Let $\AA_1, \AA_2$ be any pair of $\frak{sl} (2, \Complex)$-matrices, both with eigenvalues $\pm 1/3$, and let $\nabla$ be given by the formula \eqref{210610202346}.
Then the $\nabla$ has Levelt exponents $\pm 1/3$ at each pole.

To abelianise $\nabla$, we must choose a generic saddle-free quadratic differential on $(X, D)$ with residues $1/9$ at each point of $D$.
One such choice is the quadratic differential $\phi_1$ from \autoref{210610162950}.
Its spectral curve $\sf{\Sigma}$ was described in \autoref{210610163747}, its Stokes and spectral graphs were detailed in \autoref{210610182526}, and the relevant Stokes open cover was presented in \autoref{210610185958}.
Finally, in \autoref{210610201135}, we illustrate the abelianisation construction by displaying which Levelt line subbundle is considered on which Stokes and spectral region.\end{example}

\begin{figure}
\centering
\includegraphics[width=\textwidth]{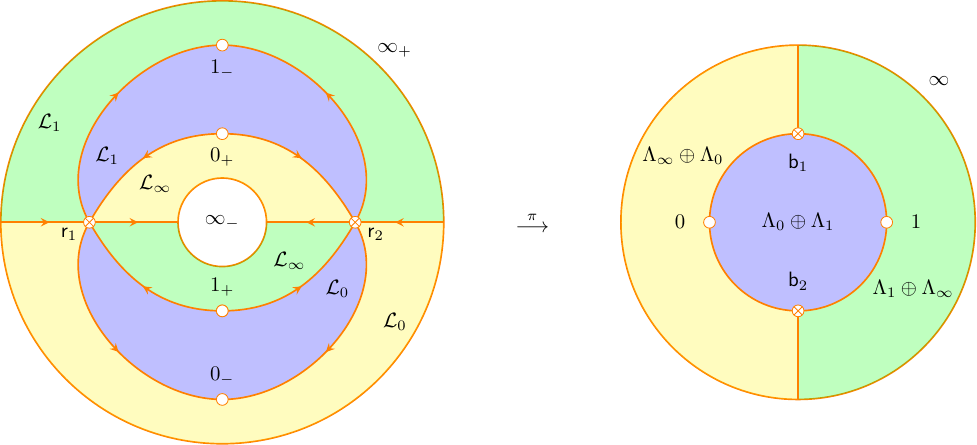}
\caption{Illustration of the construction of the abelianisation line bundle $\cal{L}$ for a connection on $(\PPP^1, \set{0,1,\infty})$ from \autoref{210610135430} using the quadratic differential $\phi_1$ from \autoref{210610162950}.}
\label{210610201135}
\end{figure}

\subsection{The Voros Cocycle}

This section introduces the main ingredient in constructing the deablianisation functor $\pi_\ab^\Gamma$, the \textit{Voros cocycle}.
Let $(\cal{L}, \de, \mu)$ be its abelianisation of $(\cal{E}, \nabla, \MM) \in \Conn_X^2 (\Gamma)$.

\paragraph{The canonical nonabelian cocycle $\VV$.}
Let $U_\alpha \in \Gamma_1$ be a Stokes ray on $X$ with polar vertex $\pp \in D$ and branch vertex $\bb \in B$.
It is a component of the intersection of exactly two Stokes regions $U_\II, U_\JJ$ (see \autoref{181129171538}).
Consider the pair of canonical identifications given by \hyperref[181127214543]{\Autoref*{181126145627}\ref*{181127214543}}:
\eqntag{\label{181107131507}
	\phi^\phantomindex_\II : \evat{\cal{E}}{U_\II} \iso \evat{\pi_\ast \cal{L}}{U_\II}
\qtext{and}
	\phi^\phantomindex_\JJ : \evat{\cal{E}}{U_\JJ} \iso \evat{\pi_\ast \cal{L}}{U_\JJ}
\fullstop
}
Over the Stokes ray $U_{\alpha}$, their ratio yields a flat automorphism of $(\pi_\ast \cal{L}, \pi_\ast \de)$:
\eqntag{\label{181127123404}
	\VV_{\alpha} \coleq \phi^\phantomindex_\JJ \circ \phi_\II^{-1} \in \Aut \big( \evat{\pi_\ast \LL}{U_{\alpha}} \big)
}
where $\pi_\ast \LL$ denotes the associated local system $\cal{ker} (\pi_\ast \de)$ on $X^\circ$.
The nerve of the cover $\frak{U}_\Gamma$ of $X^\circ$ consists of Stokes rays, so we obtain a \Cech 1-cocycle $\VV$ with values in the local system $\cal{Aut} (\pi_\ast \LL)$:
\eqntag{\label{181126185446}
	\VV \coleq \set{ \VV_{\alpha} ~\big|~ \alpha \in \Gamma_1}
		\in \check{Z}^1 \big( \frak{U}_\Gamma, \cal{Aut} ( \pi_\ast \LL) \big)
\fullstop
}

\begin{lem}{181109154245}
If $(\cal{E}, \nabla, \MM) \in \Conn_X^2 (\Gamma)$, let $(\cal{L}, \de, \mu)$ be its abelianisation, and consider the pushforward $\pi_\ast \cal{L} = \pi_\ast \pi^\ab_\Gamma \cal{E}$.
If $\VV$ is the cocycle \eqref{181126185446}, then there is a canonical isomorphism
\eqntag{
	\VV \cdot \pi_\ast \pi^\ab_\Gamma \cal{E} \iso \cal{E}
\fullstop
}
\end{lem}

\begin{proof}
The action of the cocycle $\VV$ on the pushforward bundle $\pi_\ast \cal{L}$ is a new bundle $\cal{E}' \coleq \VV \cdot \pi_\ast \cal{L}$.
Explicitly, the local piece $\cal{E}'_{\II}$ over a Stokes region $U_{\II}$ is defined to be $\evat{\pi_\ast \cal{L}}{U_{\II}}$, and the gluing data over a Stokes ray $U_{\alpha} \subset U_\II \cap U_\JJ$ is given by $\VV_{\alpha}$:
\eqntag{
\begin{tikzcd}[ampersand replacement =\&, row sep = small]
			\evat{\cal{E}'_{\II}}{U_{\alpha}}
				\ar[r, "\simlow"]
				\ar[d, equal]
	\&		\evat{\cal{E}'_{\JJ}}{U_{\alpha}}
				\ar[d, equal]
\\			\evat{\pi_\ast \cal{L}}{U_{\alpha}}
				\ar[r, "\simlow", "\VV_{\alpha}"']
	\&		\evat{\pi_\ast \cal{L}}{U_{\alpha}}
\fullstop
\end{tikzcd}
}
But this commutative square together with \eqref{181107131507} and \eqref{181127123404} imply that $\cal{E}$ and $\cal{E}'$ are canonically isomorphic.
\end{proof}

\paragraph{Transposition paths.}
Let us explicitly compute each automorphism $\VV_{\alpha}$ with respect to a pair of canonical decompositions of $\pi_\ast \cal{L}$ over the Stokes ray $U_\alpha$.
Through the isomorphisms $\evat{\pi_\ast \cal{L}}{U_\II} \iso \Lambda_i \oplus \Lambda_{i'}$ and $\evat{\pi_\ast \cal{L}}{U_\JJ} \iso \Lambda_j \oplus \Lambda_{j}$, the automorphism $\VV_{\alpha} = \phi^\phantomindex_\JJ \circ \phi_\II^{-1}$ over $U_\alpha$ is just the identity on $\cal{E}$ written as a map $\Lambda_i \oplus \Lambda_{i'} \to \Lambda_j \oplus \Lambda_{j'}$.
Notice that $\Lambda_{i'} = \Lambda_{j'}$ because they are continuations of the same line bundle germ at $\pp$, so using \eqref{190514172603} we find:
\eqntag{\label{181130084156}
	\VV_{\alpha}
	= \id_\cal{E}
	=
	\mtx{ ~1 & \Delta_\alpha ~ \\ ~\HIDE{0} & g_\alpha~}
:
\begin{tikzcd}[ampersand replacement=\&, row sep = tiny, baseline=-2.5pt]
			\Lambda_{i'}
				\ar[d, "\oplus" description]
				\ar[r, equal, shorten >=-2.5pt, shorten <=-5pt, "1"]
	\&		\Lambda_{j'}
				\ar[d, "\oplus" description]
\\			\Lambda_i
				\ar[r, shorten >=-2.5pt, shorten <=-5pt, "g_\alpha"']
				\ar[ur, shorten >=-2.5pt, shorten <=-5pt, "\Delta_\alpha" description]
	\&		\Lambda_j
\end{tikzcd}
}
Now, we can decompose the map $\Delta_\alpha : \Lambda_i \to \Lambda_{j'}$ through canonical inclusions, projections, and the upper-triangular expressions \eqref{190514172603} for the identity on $\cal{E}$ as follows:
\eqntag{
	\Delta_\alpha
		= \left(
		\Lambda_i
		\too
		\begin{tikzcd}[ampersand replacement=\&, row sep = tiny, baseline=-2.5pt]
					\Lambda_{i'}
						\ar[d, "\oplus" description]
						\ar[r, shorten >=-2.5pt, shorten <=-5pt]
			\&		\Lambda_{k'}
						\ar[d, "\oplus" description]
		\\			\Lambda_i
						\ar[r, equal, shorten >=-2.5pt, shorten <=-5pt]
						\ar[ur, shorten >=-2.5pt, shorten <=-5pt]
			\&		\Lambda_k
		\end{tikzcd}
		\too
		\Lambda_k
		\too
		\begin{tikzcd}[ampersand replacement=\&, row sep = tiny, baseline=-2.5pt]
					\Lambda_k
						\ar[d, "\oplus" description]
						\ar[r, shorten >=-2.5pt, shorten <=-5pt]
			\&		\Lambda_{j'}
						\ar[d, "\oplus" description]
		\\			\Lambda_{k'}
						\ar[r, equal, shorten >=-2.5pt, shorten <=-5pt]
						\ar[ur, shorten >=-2.5pt, shorten <=-5pt]
			\&		\Lambda_j
		\end{tikzcd}
		\too
		\Lambda_{j'}
		\right)
}
We interpret the first and second upper-triangular expressions as the identity maps on $\cal{E}$ over $U_\gamma$ and $U_\beta$, respectively.
Since all these bundle maps are $\nabla$-flat, the map $\Delta_\alpha$ can be interpreted as the endomorphism of the fibre of $\cal{E}$ over a point in $U_\alpha$ obtained as the composition of $\nabla$-parallel transports $\PP_\II, \PP_\KK, \PP_\JJ$ along paths $\delta_\II$ contained in $U_\II$ from $U_\alpha$ to $U_\gamma$, followed by $\delta_\KK$ contained in $U_\KK$ from $U_\gamma$ to $U_\beta$, followed by $\delta_\JJ$ contained in $U_\JJ$ from $U_\beta$ back to $U_\alpha$ (see \autoref{181129171538}).
\begin{figure}[t]
\centering
\includegraphics{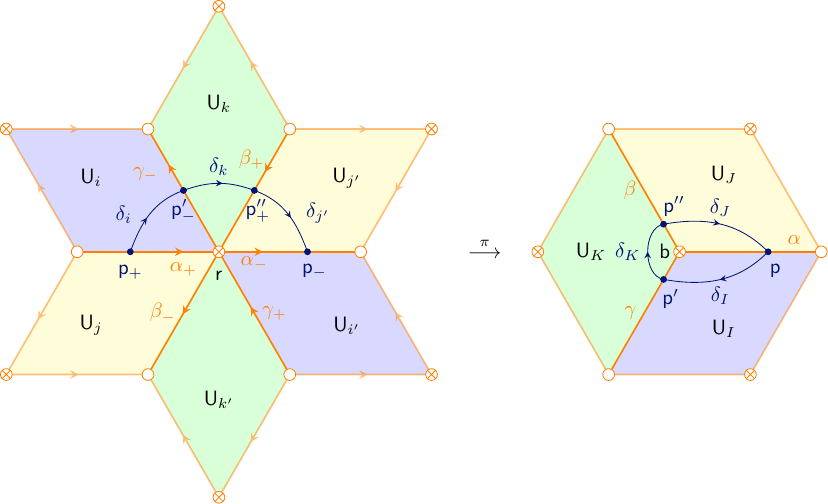}
\caption{%
$U_\II, U_\JJ, U_\KK \subset X$ are the Stokes regions with $\II = \set{i,i'}, \JJ = \set{j,j'}, \KK = \set{k,k'}$.
The stokes rays $U_\alpha, U_\gamma, U_\gamma$ are indicated by $\alpha, \beta, \gamma$ (same for the spectral rays).
$\bb \in B$ is the branch point and $\rr \in R$ is the ramification point above $\bb$.
}
\label{181129171538}
\end{figure}
Explicitly:
\eqntag{
\begin{tikzcd}[ampersand replacement =\&, column sep = small]
	\Big( ~
		\Lambda_i
			\ar[r, "\Delta_\alpha"]
\&		\Lambda_{j'}
		~	\Big)
			\ar[r, phantom, "=" description]
\&
	\Big( ~
		\Lambda_i
			\ar[r, "\PP_\II"]
\&		\Lambda_i
			\ar[r, equal, "1"]
\&		\Lambda_k
			\ar[r, "\PP_\KK"]
\&		\Lambda_k
			\ar[r, "g_{\beta}^{-1}"]
\&		\Lambda_{j'}
			\ar[r, "\PP_\JJ"]
\&		\Lambda_{j'}
		~	\Big)
\end{tikzcd}
\fullstop
}
The key idea, which goes back to Gaiotto--Moore--Neitzke \cite{MR3115984}, is to notice that this expression has an interpretation as a parallel transport for the abelian connection $\de$ on the spectral curve.
Indeed, if we fix points $\pp, \pp', \pp''$ in $U_\alpha, U_\gamma, U_\beta$ as shown in \autoref{181129171538}, then through the canonical identification of fibres using \hyperref[181127214543]{\Autoref*{181126145627}\ref*{181127214543}}, we have:
\small
\eqn{
\begin{tikzcd}[ampersand replacement =\&, column sep = small, row sep = small]
	\Big( ~
		\evat{\Lambda_i}{\pp}
			\ar[r, "\Delta_\alpha"]
			\ar[d, phantom, "\congdown" description]
\&		\evat{\Lambda_{j'}}{\pp}
			\ar[d, shift right = 10pt, phantom, "\congdown" description]
		~	\Big)
			\ar[r, phantom, "=" description]
\&
	\Big( ~
		\evat{\Lambda_i}{\pp}
			\ar[r, "\PP_\II"]
			\ar[d, phantom, "\congdown" description]
\&		\evat{\Lambda_i}{\pp}~
			\ar[r, shorten >=-5pt, shorten <=-5pt, equal, "1"]
			\ar[d, phantom, "\congdown" description]
\&		~\evat{\Lambda_k}{\pp}
			\ar[r, "\PP_\KK"]
			\ar[d, phantom, "\congdown" description]
\&		\evat{\Lambda_k}{\pp}
			\ar[r, shorten >=-5pt, shorten <=-5pt, "g_{\beta}^{-1}"]
			\ar[d, phantom, "\congdown" description]
\&		\evat{\Lambda_{j'}}{\pp}
			\ar[r, "\PP_\JJ"]
			\ar[d, phantom, "\congdown" description]
\&		\evat{\Lambda_{j'}}{\pp}
			\ar[d, shift right = 10pt, phantom, "\congdown" description]
		~	\Big)
\\
	\Big( ~
		\evat{\cal{L}}{\pp_+}
			\ar[r, "\Delta^+_\alpha"']
\&		\evat{\cal{L}}{\pp_-}
		~	\Big)
			\ar[r, phantom, "=" description]
\&
	\Big( ~
		\evat{\cal{L}}{\pp_+}
			\ar[r, "p_i"']
\&		\evat{\cal{L}}{\pp'_-} ~
			\ar[r, shorten >=-5pt, shorten <=-5pt, equal, "(g^-_{\gamma})^{-1}"']
\&		~\evat{\cal{L}}{\pp'_-}
			\ar[r, "p_k"']
\&		\evat{\cal{L}}{\pp''_+}~
			\ar[r, shorten >=-5pt, shorten <=-5pt, "(g^+_{\beta})^{-1}"']
\&		~\evat{\cal{L}}{\pp'_+}
			\ar[r, "p_{j'}"']
\&		\evat{\cal{L}}{\pp_-}
		~~	\Big)
\end{tikzcd}
\fullstop
}
\normalsize
Here, $\Delta^+_\alpha$ is defined by the diagram; we used \eqref{190523180730}, and $p_i, p_k, p_{j'}$ are $\de$-parallel transports along the paths $\delta_i, \delta_k, \delta_{j'}$ which are the lifts of $\delta_\II, \delta_\KK, \delta_\JJ$ as shown in \autoref{181129171538}.
Since $g^+_{\beta}, g^-_{\gamma}$ are precisely the gluing maps for $\cal{L}$, we find that $\Delta^+_\alpha$ is nothing but the parallel transport of $\de$ along the clockwise semicircular path $\delta_\pp^+ \coleq \delta_{j'} \delta_k \delta_{i}$ (our paths compose the same way as maps: from right to left) around the ramification point $\rr$ starting at $\pp_+$ and ending at $\pp_-$.
The Stokes graph determines such sheet transposition paths on all Stokes rays: i.e., for any $\pp \in U_\alpha$, the path $\delta_\pp^+$ on $\sf{\Sigma}$ is the unique lift starting at $\pp_+$ of a clockwise loop $\delta_\pp$ based at $\pp \in U_\alpha$ around the branch point $\bb$ (see \autoref{181102111720}).
\begin{figure}[t]
\centering
\includegraphics{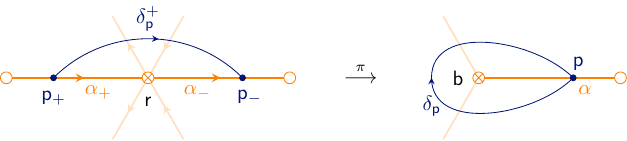}
\caption{The sheet transposition path $\delta_\pp^+$ associated with the positive spectral ray $\alpha_+$.
Its projection onto $X$ is a clockwise loop $\delta_\pp$ around the branch point $\bb$.}
\label{181102111720}
\end{figure}

\begin{lem}{181127112030}
For every Stokes ray $U_\alpha \subset X$ and every point $\pp \in U_{\alpha}$, the automorphism $\VV_{\alpha, \pp}$ of the fibre $\evat{\pi_\ast \cal{L}}{\pp}$ is:
\eqntag{\label{181004164759}
		\VV_{\alpha,\pp}
	=
	\mtx{ ~1 & \Delta_{\alpha,\pp}^+~ \\ ~\HIDE{0} & 1~}
:
\begin{tikzcd}[ampersand replacement=\&, row sep = small, baseline=-2.5pt]
			\evat{\cal{L}}{\pp_-}
				\ar[d, "\oplus" description]
				\ar[r, equal, shorten >=-2.5pt, shorten <=-5pt]
	\&		\evat{\cal{L}}{\pp_-}
				\ar[d, "\oplus" description]
\\			\evat{\cal{L}}{\pp_+}
				\ar[ur, shorten >=-2.5pt, shorten <=-5pt, "\smash{\Delta_{\alpha,\pp}^+}" description]
				\ar[r, equal, shorten >=-2.5pt, shorten <=-5pt]
	\&		\evat{\cal{L}}{\pp_+}
\end{tikzcd}
\fullstop{,}
\qquad
	\Delta_{\alpha,\pp}^+ \coleq \Par \big(\de, \delta^+_{\pp} \big)
\fullstop
}
\end{lem}

The correspondence $\pp_+ \mapsto \delta^+_{\pp}$ is a well-defined map $\delta^+_\alpha : U^+_\alpha \to \mathsf{\Pi}_1 (\sf{\Sigma}^\circ)$, where $\mathsf{\Pi}_1 (\sf{\Sigma}^\circ)$ is the fundamental groupoid of the punctured spectral curve, which is the set of paths on $\sf{\Sigma}^\circ = \sf{\Sigma} \setminus R$ considered up to homotopy with fixed endpoints.
If we define a flat bundle isomorphism 
\eqntag{
	\Delta_\alpha^+ \coleq \Par (\de, \delta^+_\alpha) : 
		\evat{\cal{L}}{U^+_\alpha} \iso \sigma^\ast \evat{\cal{L}}{U^+_\alpha}
\fullstop{,}
}	
then $\Delta_\alpha = \pi_\ast \Delta^+_\alpha$ defines an endomorphism of $\pi_\ast \cal{L}$ over the Stokes ray $U_{\alpha}$.
So \Autoref{181127112030} may be expressed in terms of bundle maps as follows.

\begin{lem}{181129205829}
For every $\alpha \in \Gamma_1$, the automorphism $\VV_{\alpha}$ of $\evat{\pi_\ast \cal{L}}{U_\alpha}$ is $\VV_{\alpha} = \id + \pi_\ast \Delta^+_\alpha$.
\end{lem}

\paragraph{The Voros cocycle.}
One of the central observations in this paper is that formula \eqref{181004164759} does not depend on the fact that $(\cal{L}, \de)$ is the abelianisation of $(\cal{E}, \nabla)$.
Indeed, this formula is written purely in terms of the parallel transport along canonically defined paths on $\sf{\Sigma}^\circ$ and the pushforward functor $\pi_\ast$.
In other words, if $(\cal{L}, \de) \in \Conn_\sf{\Sigma}^1$ is \textit{any} abelian connection (i.e., not a priori the abelianisation of some connection on $X$), then for each Stokes ray $\alpha \in \Gamma_1$, we can consider the automorphism $\VV_{\alpha}$ of $\pi_\ast \cal{L}$ over $U_{\alpha}$ defined by
\eqntag{\label{181130121454}
		\evat{\VV_{\alpha}}{\pp}
	=
	\mtx{ ~1 & \evat{\Delta^+_{\alpha}}{\pp_+} ~ \\ ~\HIDE{0} & 1~}
:
\begin{tikzcd}[ampersand replacement=\&, row sep = small, baseline=-2.5pt]
			\evat{\cal{L}}{\pp_-}
				\ar[d, "\oplus" description]
				\ar[r, equal, shorten >=-2.5pt, shorten <=-5pt]
	\&		\evat{\cal{L}}{\pp_-}
				\ar[d, "\oplus" description]
\\			\evat{\cal{L}}{\pp_+}
				\ar[ur, shorten >=-2.5pt, shorten <=-5pt, "\smash{\Delta^+_{\alpha}}" description]
				\ar[r, equal, shorten >=-2.5pt, shorten <=-5pt]
	\&		\evat{\cal{L}}{\pp_+}
\end{tikzcd}
\fullstop{,}
\qquad
	\evat{\Delta^+_{\alpha}}{\pp_+} \coleq \Par \big(\de, \evat{\delta^+_{\alpha}}{\pp_+} \big)
\fullstop{,}
}
for each $\pp \in U_{\alpha}$ with preimages $\pp_\pm \in U_{\alpha}^\pm$.
As a bundle automorphism over $U_{\alpha}$,
\eqntag{\label{181130121504}
	\VV_{\alpha} = \id + \pi_\ast \Delta^+_\alpha
	\in \Aut \big( \pi_\ast \LL_{\alpha} \big)
\fullstop{,}
}
where $\pi_\ast \LL \coleq \cal{ker} (\pi_\ast \de)$ and $\pi_\ast \LL_{\alpha} \coleq \evat{\pi_\ast \LL}{U_{\alpha}}$.
This yields a cocycle
\eqntag{\label{181127115041}
	\VV \coleq \set{ \VV_{\alpha} ~\big|~ \alpha \in \Gamma_1}
		\in \check{Z}^1 \Big(\frak{U}_{\Gamma}, \cal{Aut} ( \pi_\ast \LL) \Big)
\fullstop
}
Now, if $\varphi : (\cal{L}, \de) \too (\cal{L}', \de')$ is a morphism in $\Conn_\sf{\Sigma}^1$, and $\VV, \VV'$ are respectively the cocycles for $\cal{L}, \cal{L}'$ defined by the formula \eqref{181130121454}, then the identity $\de \varphi = \varphi \de'$ immediately implies the following commutative square for every $\alpha$:
\eqntag{\label{181103130752}
\begin{tikzcd}[ampersand replacement = \&]
			\pi_\ast \LL_{\alpha}
					\ar[r, "\VV_{\alpha}"]
					\ar[d, "\pi_\ast \varphi"']
	\&		\pi_\ast \LL_{\alpha}
					\ar[d, "\pi_\ast \varphi"]
\\			\pi_\ast \LL'_{\alpha}
					\ar[r, "\VV'_{\alpha}"']
	\&		\pi_\ast \LL'_{\alpha}
\end{tikzcd}
\fullstop
}
In other words, for every Stokes ray $\alpha \in \Gamma_1$, the collection 
\eqntag{\label{181127120526}
	\mathbb{V}_{\alpha}
			\coleq \set{\Big. \VV_{\alpha} \in \Aut \big(\pi_\ast \LL_{\alpha} \big) }_{(\cal{L}, \de)}
\fullstop{,}
}
indexed by abelian connections $(\cal{L}, \nabla) \in \Conn_\sf{\Sigma}^1$, forms a natural transformation
\eqntag{
	\mathbb{V}_{\alpha} : \pi_\ast \Rightarrow \pi_\ast
\fullstop
}
of the pushforward functor \eqref{181126184346}, defined over $U_{\alpha}$.
We obtain a cocycle valued in the local system $\cal{Aut} (\pi_\ast)$ of nonabelian groups on the punctured base curve $X^\circ$ consisting of natural automorphisms of $\pi_\ast$.

\begin{defn}{181127120022}
The \dfn{Voros cocycle} is the nonabelian \Cech $1$-cocycle
\eqntag{
	\Voros 
		\coleq \set{\Voros_{\alpha} ~\big|~ \alpha \in \Gamma_1}
		\in \check{Z}^1 \Big(\frak{U}_\Gamma, \cal{Aut} (\pi_\ast) \Big)
\fullstop
\tag*{\qedhere}
}
\end{defn}

\paragraph{Abelianisation of the Voros cocycle.}
The parallel transports $\Delta_\alpha$ can also be arranged into a cocycle as follows.
If $(\cal{L}, \de) \in \Conn_\sf{\Sigma}^1$ is any abelian connection, then $\Delta_\alpha^+ = \Par (\de, \delta_\alpha^+) \in \Hom (\LL_\alpha^+, \LL_\alpha^-) = \Hom (\LL_\alpha^+, \sigma^\ast \LL_\alpha^+)$, where $\LL \coleq \cal{ker} (\de)$ and  $\LL_\alpha^\pm \coleq \evat{\LL}{U^\pm_{\alpha}}$ for each $\alpha \in \Gamma^+_1$.
The sheaf $\cal{Hom} (\LL, \sigma^\ast \LL)$ is a local system of \textit{abelian} groups, and we can define an abelian \Cech $1$-cocycle on $\sf{\Sigma}^\circ$ by
\eqntag{\label{181127123642}
	\Delta \coleq \set{\Delta_\alpha^+, \Delta_\alpha^- ~\big|~ \pm \alpha \in \vec{\Gamma}_1^\pm}
	\in \check{Z}^1 \Big(\frak{U}_\vec{\Gamma}, \cal{Hom} (\LL, \sigma^\ast \LL) \Big)
\fullstop{,}
}
by $\Delta^+_\alpha \coleq \Par (\de, \delta^+_\alpha)$ and $\Delta_{\alpha}^- \coleq 0$.
If $\varphi : (\cal{L}, \de) \too (\cal{L}', \de')$ is a morphism in $\Conn_\sf{\Sigma}^1$, and $\Delta, \Delta'$ are the corresponding cocycles, then the identity $\de \varphi = \varphi \de'$ implies for every $\alpha$ a pair of commutative squares:
\eqntag{\label{191116095835}
\begin{tikzcd}[ampersand replacement = \&]
			\LL_\alpha^\pm
					\ar[r, "\Delta_\alpha^\pm"]
					\ar[d, "\varphi"']
	\&		\sigma^\ast \LL_\alpha^\pm
					\ar[d, "\sigma^\ast \varphi"]
\\			\LL_\alpha^{\prime\pm}
					\ar[r, "\Delta^{\prime\pm}_\alpha"']
	\&		\sigma^\ast \LL_\alpha^{\prime\pm}
\end{tikzcd}
\fullstop
}
In other words, for every $\alpha$, the collection of flat homomorphisms 
\eqntag{
	\bbDelta_{\alpha}^\pm
		\coleq \set{\Big. \Delta_{\alpha}^\pm \in \Hom (\LL_{\alpha}^\pm, \sigma^\ast \LL_{\alpha}^\pm)}_{(\cal{L}, \de)}
\fullstop{,}
}
indexed by abelian connections $(\cal{L}, \de) \in \Conn_\sf{\Sigma}^1$, forms a natural transformation
\eqntag{
	\bbDelta_{\alpha}^\pm : \id \implies \sigma^\ast
\fullstop{,}
}
defined over $U_{\alpha}^\pm$.
Here, $\sigma^\ast : \Conn_\sf{\Sigma}^1 \to \Conn_\sf{\Sigma}^1$ is the pullback functor by the canonical involution $\sigma$.
Thus, we obtain a cocycle valued in the local system $\cal{Hom} (\id, \sigma^\ast)$ of abelian groups on the punctured spectral curve $\sf{\Sigma}^\circ$ consisting of natural transformations from the identity functor $\id$ to the pullback functor $\sigma^\ast$:
\eqntag{\label{190129125956}
	\bbDelta
		\coleq \set{\mathbbold{\Delta}_{\alpha}^\pm ~\big|~ \alpha \in \Gamma_1}
		\in \check{Z}^1 \Big(\frak{U}_\vec{\Gamma}, \cal{Hom} (\id, \sigma^\ast) \Big)
\fullstop
}
Formula \eqref{181130121504} makes it apparent that the Voros cocycle $\Voros$ is completely determined by the cocycle $\bbDelta$; let us make this precise.
Suppose $(\cal{L}, \de) \in \Conn_\sf{\Sigma}^1$, and choose a point $\pp \in U_{\alpha}$ for some $\alpha$.
If $\pp_\pm \in U_{\alpha}^\pm$ are the two preimages of $\pp$, then the canonical isomorphism $\pi_\ast \LL_\pp \iso \LL_{\pp_-} \oplus \LL_{\pp_+}$ on stalks induces a canonical inclusion of $\cal{Hom} (\LL, \sigma^\ast \LL)_{\pp_\pm}$ into $\cal{End} (\pi_\ast \LL)_\pp$ via
\eqn{
	\cal{Hom} (\LL, \sigma^\ast \LL)_{\pp_\pm}
		= \Hom (\LL_{\pp_\pm}, \sigma^\ast \LL_{\pp_\pm})
		= \Hom (\LL_{\pp_\pm}, \LL_{\pp_\mp})
		\inj \End (\pi_\ast \LL_\pp )
		= \cal{End} (\pi_\ast \LL)_\pp
\fullstop
}
Given any $c \in \cal{Hom} (\LL, \sigma^\ast \LL)_{\pp_\pm}$, we denote its image in $\cal{End} (\pi_\ast \LL)_\pp$ by $\pi_\ast c$.

\begin{prop}{181102200623}
The Voros cocycle $\Voros$ and the abelian cocycle $\bbDelta$ satisfy 
\eqntag{
	\mathbb{V} = \bbid + \pi_\ast \mathbbold{\Delta}
\fullstop{,}
}
where $\bbid$ is the identity cocycle.
\end{prop}

That is to say, the nonabelian Voros cocycle $\Voros$ is actually `in disguise' the data of an \textit{abelian} cocycle $\bbDelta$ but on a different curve.
In other words, $\bbDelta$ should be thought of as the \textit{abelianisation} of the Voros cocycle.

\begin{proof}
Notice that $\pi$ induces a double cover $\dot{\frak{U}}_\vec{\Gamma} \to \dot{\frak{U}}_\Gamma$, yielding a map on cocycles:
\eqntag{\label{181127132733}
	\check{Z}^1 \Big(\frak{U}_\vec{\Gamma}, \cal{Hom} (\id, \sigma^\ast) \Big)
		\too
	\check{Z}^1 \Big(\frak{U}_\Gamma, \cal{Aut} (\pi_\ast) \Big)
\qtext{given by}
	c \mapstoo \bbid + \pi_\ast c
\fullstop
}
Then formula \eqref{181130121504} implies that $\Voros$ is the image of $\bbDelta$.
\end{proof}

\subsection{The Nonabelianisation Functor}
\label{181130123445}
In this section, we construct the nonabelianisation functor $\pi_\ab^\Gamma$ and prove that it is an inverse equivalence to the abelianisation functor $\pi_\Gamma^\ab$.
The main ingredient is the Voros cocycle $\Voros$, and the construction proceeds in two steps.
If $(\cal{L}, \de)$ is an abelian connection on $\sf{\Sigma}$, we first use the the pushforward functor $\pi_\ast$ to obtain a rank-two connection $(\pi_\ast \cal{L}, \pi_\ast \de)$ on $(X, D \cup B)$.
But $\pi_\ast \de$ does \textit{not} holomorphically extend over the branch locus $B$, because it has nontrivial monodromy around $B$, as we remarked after the proof of \Autoref{181116135914}.
Therefore, $\pi_\ast$ cannot invert $\pi^\ab_\Gamma$, because its image is not even contained in $\Conn_X^2$.
Instead, step two is to use the Voros cocycle $\Voros$ to deform $\pi_\ast$ as a functor.
The result is the nonabelianisation functor $\pi_\ab^\Gamma$.

\paragraph{Construction of $\nabla$.}
Given any abelian connection $(\cal{L}, \de, \mu) \in \Conn_\sf{\Sigma}^1$, we construct $(\cal{E}, \nabla, \MM) \in \Conn_X^2 (\Gamma)$.
Consider the pushforward $(\pi_\ast \cal{L}, \pi_\ast \de, \pi_\ast \mu)$.
The Voros cocycle $\Voros$ determines a cocycle $\VV \coleq \Voros (\cal{L}) \in \check{Z}^1 \big(\frak{U}_\Gamma, \cal{Aut} (\pi_\ast \cal{L}) \big)$.

\paragraph{Definition over Stokes regions.}
The main step in the construction is to use $\VV$ to reglue $\pi_\ast \cal{L}$ over Stokes rays.
For each Stokes region $U_\II$, let 
\eqntag{
	\cal{E}_\II \coleq \evat{\pi_\ast \cal{L}}{U_{\II}},
\qquad
	\nabla_\II \coleq \evat{\pi_\ast \de}{U_{\II}},
\qquad
	\MM_\II \coleq \evat{\pi_\ast \mu}{U_{\II}},
}
and if $U_\alpha$ is a Stokes ray in the ordered double intersection $U_\II \cap U_\JJ$, then the gluing over $U_\alpha$ is given by $\VV_\alpha : \cal{E}_\II \iso \cal{E}_\JJ$.
If $U_i$ is a spectral region in the preimage of $U_\II$, then since $\cal{E}_\II (U_\II) = \cal{L} (U_i) \oplus \sigma^\ast \cal{L} (U_i)$, the map $\MM_\II$ defines an $\frak{sl}_2$-structure on each local piece $\cal{E}_{\II}$.
Moreover, $\MM_\II$ and $\MM_\JJ$ glue over $U_\alpha$ because $\VV_{\alpha}$ is unipotent with respect to the corresponding decompositions.

\paragraph{Definition at the poles.}
Recall that the infinitesimal punctured disc $U_\pp^\ast$ centred at a point $\pp \in D$ is covered by sectorial neighbourhoods coming from the Stokes regions incident to $\pp$.
Thanks to the upper-triangular nature of the Voros cocycle $\VV$, we obtain a flat bundle $\cal{E}_\pp^\ast$ over $U_\pp^\ast$ equipped with a filtration $(\cal{E}_\pp^\ast)^\bullet$ whose associated graded is canonically isomorphic to $\evat{\pi_\ast \cal{L}}{U_{\pp}^\ast}$.
Now, it is a simple fact that if the associated graded of a filtered connection extends over a point, then the filtered connection itself extends with the same Levelt exponents.
Thus, $\cal{E}_\pp^\ast$ has a canonical extension over $U_\pp$ to a bundle $\cal{E}_\pp$ with connection $\nabla_\pp$ that has logarithmic poles at $\pp$ and Levelt exponents $\pm \lambda_\pp$.
It remains to define $\nabla$ over the branch locus $B$.

\paragraph{Definition at the branch points.}
We will first compute the monodromy of $\nabla$ around each branch point directly to show that it is trivial, and then use Deligne's canonical extension \cite[pp.91-96]{MR0417174}.

\begin{lem}{181109134108}
The monodromy of $\nabla$ around any branch point is trivial.
Therefore, the connection $(\cal{E}, \nabla, \MM)$ on $X^\circ$ has a canonical holomorphic extension over $B$.
\end{lem}

The technique is to express the parallel transport of $\nabla$ along paths on $X$ in terms of the parallel transport of $\de$ along their lifts to $\sf{\Sigma}$ as well as the sheet transposition paths.
We adopt the following notation for the parallel transports of $\nabla, \de, \pi_\ast \de$, respectively:
\eqn{
	\PP : \sf{\Pi}_1 (X^\circ) \to \GL (\cal{E}),
\qquad p : \sf{\Pi}_1 (\sf{\Sigma}^\circ) \to \GL (\cal{L}),
\qquad \pi_\ast p : \sf{\Pi}_1 (X^\circ) \to \GL (\pi_\ast \cal{L})
\fullstop
}
It follows immediately from the construction of $\cal{E}$ that if $\wp$ is a path on $X^\circ$ contained in a Stokes region, then $\PP (\wp) = \pi_\ast p (\wp)$.
Explicitly, let $\wp', \wp''$ be the two lifts of $\wp$ to $\sf{\Sigma}$.
Let $\xx, \yy$ be the startpoint and the endpoint of $\wp$, and similarly for $\wp', \wp''$.
Then, for example, the fibre $E_\xx = \evat{\cal{E}}{\xx}$ is the direct sum of fibres $L_{\xx'} \oplus L_{\xx''}$ of $\cal{L}$.
With respect to these decompositions, the parallel transport $\PP (\wp) : E_{\xx} \too E_{\yy}$ is expressed as
\eqntag{\label{181108114912}
	\PP (\wp) 
		= \pi_\ast p (\wp)
		= \mtx{ p (\wp') & \HIDE{0} \\ \HIDE{0} & p (\wp'') }
:
\begin{tikzcd}[ampersand replacement=\&, row sep = small, baseline=-2.5pt]
			L_{\xx'}
				\ar[d, "\oplus" description]
				\ar[r, shorten >=-2.5pt, shorten <=-5pt]
	\&		L_{\yy'}
				\ar[d, "\oplus" description]
\\			L_{\xx''}
				\ar[r, shorten >=-2.5pt, shorten <=-5pt]
	\&		L_{\yy''}
\end{tikzcd}
\fullstop
}

We say that a path $\wp$ on $X^\circ$ (or $\sf{\Sigma}^\circ$) is a \dfn{short path} if its endpoints do not belong to the Stokes graph $\Gamma$ (or to the spectral graph $\vec{\Gamma}$) and it intersects at most one Stokes ray (or spectral ray).
If $\wp$ is a short path on $X^\circ$ that intersects a Stokes ray ${\alpha} \in \Gamma_1$, then $\wp$ is divided into two segments $\wp_-, \wp_+$ (\autoref{181109141925}).
\begin{figure}[t]
\centering
\includegraphics{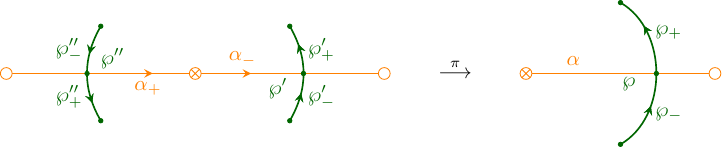}
\caption{A short path $\wp$ on $X$ intersecting the Stokes ray $\alpha$ and its lifts $\wp', \wp''$ to $\sf{\Sigma}$.}
\label{181109141925}
\end{figure}
Each $\wp_\pm$ is contained in a Stokes region, so $\PP (\wp_\pm) = \pi_\ast p (\wp_\pm)$.
On the other hand, the vector bundle $\cal{E}$ is constructed by gluing $\pi_\ast \cal{L}$ to itself over $U_{\alpha}$ by the automorphism $\VV_{\alpha}$, so we obtain the following formula for $\PP (\wp)$:
\eqntag{\label{181107185940}
	\PP (\wp) = \pi_\ast p (\wp_+) \cdot \VV_{\alpha} \cdot \pi_\ast p (\wp_-)
\fullstop
}
Explicitly, let $\wp', \wp''$ denote the two lifts of $\wp$ to $\sf{\Sigma}$, where $\wp'$ intersects $\alpha_-$ and $\wp''$ intersects $\alpha_+$ (\autoref{181109141925}).
The parallel transport $\PP (\wp) : E_{\xx} \too E_{\yy}$ can be expressed as
\eqns{
	\PP (\wp)
		= 	\mtx{ p (\wp'_+) & \HIDE{0} \\ \HIDE{0} & p (\wp''_+)}
		\mtx{ 1 & \Delta_{\alpha}^+ \\ \HIDE{0} & 1}
		\mtx{ p (\wp'_-) & \HIDE{0} \\ \HIDE{0} & p (\wp''_-)}
		= \mtx{ p (\wp') & p (\wp'_+) \Delta_{\alpha}^+ p (\wp''_-) \\ \HIDE{0} & p (\wp'')}
\fullstop
}
The off-diagonal term $p (\wp'_+) \Delta_{\alpha}^+ p (\wp''_-)$ is the parallel transport of $\de$ along the concatenated path $\wp^+_\alpha \coleq \wp'_+ \delta^+_\alpha \wp''_-$ (\autoref{181109152458}), so
\eqntag{\label{181108134016}
	\PP (\wp)
		= \mtx{ p (\wp') & p ( \wp^+_\alpha ) \\ \HIDE{0} & p (\wp'')}
:
\begin{tikzcd}[ampersand replacement=\&, row sep = tiny, baseline=-2.5pt]
			L_{\xx'}
				\ar[d, "\oplus" description]
				\ar[r, shorten >=-2.5pt, shorten <=-5pt]
	\&		L_{\yy'}
				\ar[d, "\oplus" description]
\\			L_{\xx''}
				\ar[ur, shorten >=-2.5pt, shorten <=-5pt]
				\ar[r, shorten >=-2.5pt, shorten <=-5pt]
	\&		L_{\yy''}
\end{tikzcd}
\fullstop
}
\begin{figure}[t]
\centering
\includegraphics[width=\textwidth]{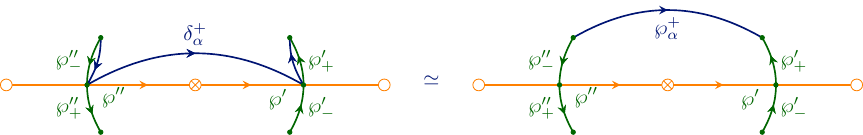}	
\caption{The concatenated path $\wp'_+ \delta^+_\alpha \wp''_-$ (left) is homotopic to $\wp^+_\alpha$ (right).}
\label{181109152458}
\end{figure}

\begin{proof}[Proof of \Autoref{181109134108}.]
Fix a branch point $\bb \in B$, and let $U_{\alpha}, U_{\beta}, U_{\gamma}$ be the three Stokes rays incident to $\bb$.
Fix a basepoint $\xx$ in the Stokes region $U_\II$ as shown in \autoref{181108130914}, and also fix a loop $\wp$ around $\bb$.
We calculate the monodromy $\PP (\wp)$.
Fix two more basepoints $\yy, \zz$ in the other two Stokes regions, thus dividing the loop $\wp$ into three short paths denoted by $\wp_\alpha, \wp_\beta, \wp_\gamma$, as explained in \autoref{181108130919}.
Then $\PP (\wp) = \PP (\wp_\gamma) \PP (\wp_\beta) \PP (\wp_\alpha)$.
Each $\PP (\wp_\bullet)$ (where $\bullet = \alpha, \beta, \gamma$) can be expressed via \eqref{181107185940} as
\eqntag{
	\PP (\wp_\bullet)
	= 	\pi_\ast p (\wp_{\bullet+}) \cdot \VV_{(\bullet)} \cdot \pi_\ast p (\wp_{\bullet-})
\fullstop
}
Now, let $\wp', \wp''$ be the two lifts of $\wp$ to $\sf{\Sigma}$, as explained in \autoref{181108131719}.
The lifts $\wp''_\alpha, \wp'_\beta, \wp''_\gamma$ intersect the positive spectral rays $\alpha_+, \beta_+, \gamma_+$, giving rise to three sheet transposition paths $\wp^+_\alpha, \wp^+_\beta, \wp^+_\gamma$ as shown in \autoref{181108140946}.
By inspection,
\eqntag{\label{181108141120}
	\wp^+_\alpha = (\wp'_\gamma \wp'_\beta)^{-1}
\fullstop{,}
\qquad
	\wp^+_\beta = (\wp'_\alpha \wp''_\gamma)^{-1}
\fullstop{,}
\qquad
	\wp^+_\gamma= (\wp''_\beta \wp''_\alpha)^{-1}
\fullstop
}
\begin{figure}[p]
\begin{minipage}[t]{0.45\linewidth}
\centering
\includegraphics{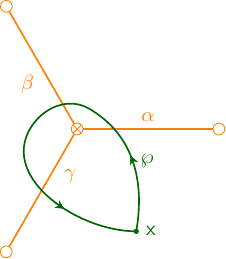}
\caption{Three Stokes rays ${\alpha}, {\beta}, {\gamma}$ on $X$ incident to the branch point $\bb \in B$, and an anti-clockwise loop $\wp$ around $\bb$ based at $\xx$.}
\label{181108130914}
\end{minipage}
\hspace{0.5cm}
\begin{minipage}[t]{0.45\linewidth}
\centering
\includegraphics{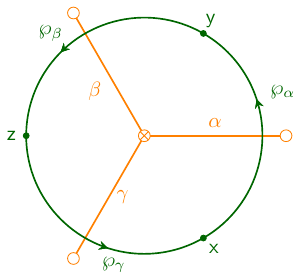}
\caption{The loop $\wp$ from \autoref{181108130914} is homotopic to the concatenated path $\wp_\gamma \wp_\beta \wp_\alpha$ as shown.}
\label{181108130919}
\end{minipage}
\end{figure}
\begin{figure}[p]
\centering
\includegraphics{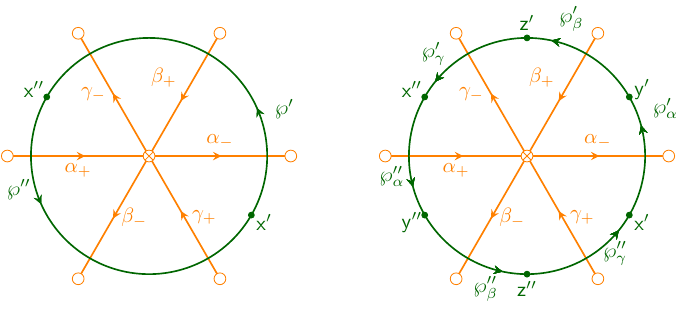}
\caption{\textit{Left:} Let $\xx', \xx''$ be the two preimages of $\xx$ on $\sf{\Sigma}$ as shown.
\textit{Right:} Let $\yy', \yy'', \zz', \zz''$ be the lifts of $\yy, \zz$ as shown, $\wp' = \wp'_\gamma \wp'_\beta \wp'_\alpha$ and $\wp'' = \wp''_\gamma \wp''_\beta \wp''_\alpha$.
}
\label{181108131719}
\end{figure}
\begin{figure}[p]
\centering
\includegraphics{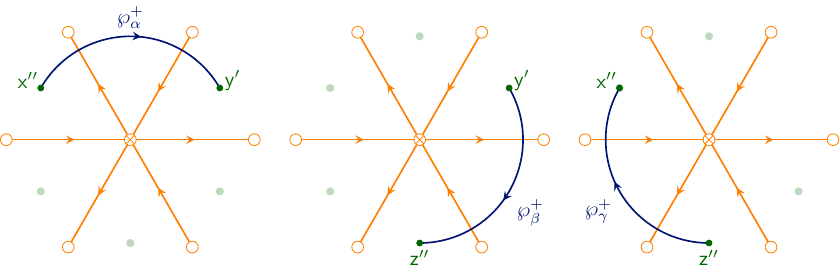}
\caption{Three sheet transposition paths $\wp^+_\alpha, \wp^+_\beta, \wp^+_\gamma$ arising from the intersections of $\wp''_\alpha, \wp'_\beta, \wp''_\gamma$ with positive spectral rays $\alpha, \beta, \gamma$, respectively.}
\label{181108140946}
\end{figure}
The explicit formula \eqref{181108134016} gives three expressions:
\eqnstag{
	\PP (\wp_\alpha)
		&= \mtx{ p (\wp'_\alpha) & p (\wp^+_{\alpha}) \\ \HIDE{0} & p (\wp''_\alpha)}
:
\begin{tikzcd}[ampersand replacement=\&, row sep = tiny, baseline=-2.5pt]
			L_{\xx'}
				\ar[d, "\oplus" description]
				\ar[r, shorten >=-2.5pt, shorten <=-5pt]
	\&		L_{\yy'}
				\ar[d, "\oplus" description]
\\			L_{\xx''}
				\ar[ur, shorten >=-2.5pt, shorten <=-5pt]
				\ar[r, shorten >=-2.5pt, shorten <=-5pt]
	\&		L_{\yy''}
\end{tikzcd}
\fullstop{,}
\\		\PP (\wp_\beta)
		&= \mtx{ p (\wp'_\beta) & \HIDE{0} \\ p (\wp^+_{\beta}) & p (\wp''_\beta)}
:
\begin{tikzcd}[ampersand replacement=\&, row sep = tiny, baseline=-2.5pt]
			L_{\yy'}
				\ar[d, "\oplus" description]
				\ar[r, shorten >=-2.5pt, shorten <=-5pt]
				\ar[dr, shorten >=-2.5pt, shorten <=-5pt]
	\&		L_{\zz'}
				\ar[d, "\oplus" description]
\\			L_{\yy''}
				\ar[r, shorten >=-2.5pt, shorten <=-5pt]
	\&		L_{\zz''}
\end{tikzcd}
\fullstop{,}
\\	\PP (\wp_\gamma)
		&= \mtx{ p (\wp'_\gamma) & p (\wp^+_{\gamma}) \\ \HIDE{0} & p (\wp''_\gamma)}
:
\begin{tikzcd}[ampersand replacement=\&, row sep = tiny, baseline=-2.5pt]
			L_{\zz'}
				\ar[d, "\oplus" description]
				\ar[r, shorten >=-2.5pt, shorten <=-5pt]
	\&		L_{\xx''}
				\ar[d, "\oplus" description]
\\			L_{\zz''}
				\ar[ur, shorten >=-2.5pt, shorten <=-5pt]
				\ar[r, shorten >=-2.5pt, shorten <=-5pt]
	\&		L_{\xx'}
\end{tikzcd}
\fullstop
}
Notice that $\PP (\wp_\beta)$ is \textit{lower}-triangular in the given decompositions of $\evat{\cal{\pi_\ast \cal{L}}}{\yy}$ and $\evat{\cal{\pi_\ast \cal{L}}}{\zz}$, because it is the lift $\wp'_\beta$ of $\wp_\beta$ starting at $\yy'$ that intersects the positive spectral ray $\beta_+$.
Also notice that the source fibre of $\PP (\wp_\alpha)$ is decomposed as $L_{\xx'} \oplus L_{\xx''}$, whilst the target fibre of $\PP (\wp_\gamma)$ is decomposed as $L_{\xx''} \oplus L_{\xx'}$, so the monodromy $\PP (\wp) \in \Aut \big(L_{\xx'} \oplus L_{\xx''} \big)$ is given by
\eqns{
	\PP (\wp)
	&= 	
		\mtx{ \HIDE{0} & 1 \\ 1 & \HIDE{0}}
		\mtx{ p (\wp'_\gamma) & p (\wp^+_{\gamma}) \\ \HIDE{0} & p (\wp''_\gamma)}
		\mtx{ p (\wp'_\beta) & \HIDE{0} \\ p (\wp^+_{\beta}) & p (\wp''_\beta)}
		\mtx{ p (\wp'_\alpha) & p (\wp^+_{\alpha}) \\ \HIDE{0} & p (\wp''_\alpha)}
\\	&=
		\mtx{ 	p (\wp''_\gamma \wp^+_{\beta} \wp'_\alpha) 
				& 	p (\wp''_\gamma \wp^+_{\beta} \wp^+_{\alpha}) + p (\wp''_\gamma \wp''_\beta \wp''_\alpha)
				\\	p( \wp'_\gamma \wp'_\beta \wp'_\alpha) + p (\wp^+_{\gamma} \wp^+_{\beta} \wp'_\alpha)
				&	p(\wp'_\gamma \wp'_\beta \wp^+_{\alpha}) + p(\wp^+_{\gamma} \wp^+_{\beta} \wp^+_{\alpha}) + p(\wp^+_{\gamma} \wp''_\beta \wp''_\alpha)}
\fullstop
}
Applying relations \eqref{181108141120}, we find that $\wp''_\gamma \wp^+_{\beta} \wp'_\alpha = \wp''_\gamma (\wp'_\alpha \wp''_\gamma)^{-1} \wp'_\alpha = 1$, which is a constant path at $\xx'$, so the top-left entry of $\PP (\wp)$ is $1$.
Next, the path $\wp^+_{\gamma} \wp^+_{\beta} \wp'_\alpha$ appearing in the bottom-left entry, simplifies to $(\wp''_\gamma \wp''_\beta \wp''_\alpha)^{-1}$, so $p(\wp^+_{\gamma} \wp^+_{\beta} \wp^+_{\alpha}) = p (\wp'_\gamma \wp'_\beta \wp'_\alpha)^{-1}$.
Now, $\wp''_\gamma \wp''_\beta \wp''_\alpha \wp'_\gamma \wp'_\beta \wp'_\alpha$ is a loop around the ramification point $\rr$ based at $\xx'$, and since the connection $\de$ has monodromy $-1$ around $\rr$ by \Autoref{181116155648}, we find:
\eqntag{
	p (\wp''_\gamma \wp''_\beta \wp''_\alpha \wp'_\gamma \wp'_\beta \wp'_\alpha) = -1
\fullstop
}
It follows that $p (\wp'_\gamma \wp'_\beta \wp'_\alpha)^{-1} = - p (\wp''_\gamma \wp''_\beta \wp''_\alpha)$, and so the bottom-left entry of $\PP (\wp)$ is $0$.
Similarly, we can calculate the other entries of $\PP (\wp)$ and find that $\PP (\wp) = \id$.
\end{proof}

\paragraph{Diagonal decompositions and transversality.}
\label{191116103716}
The fact that the connection $\nabla$ is transverse with respect to $\Gamma$ is deduced from the fact that the local and semilocal diagonal decompositions of $\cal{E}$ (\Autoref{181114180038} and \Autoref{181109180145}) can be easily recovered from our construction as follows.
Let $U_\pp$ be the infinitesimal disc around a pole $\pp \in D$.
If $U_\pp^\pm$ are respectively the infinitesimal discs around $\pp_\pm$, let $\smash{\cal{L}_\pp^\pm \coleq \evat{\cal{L}}{U_{\pp}^\pm}}$ and $\Lambda_\pp^\pm \coleq \pi_\ast \cal{L}_\pp^\pm$.
Then it follows from the construction of $\cal{E}$ over $U_\pp$ that the local diagonal decomposition of $\cal{E}_\pp$ is precisely $\evat{\pi_\ast \cal{L}}{U_{\pp}} = \Lambda_\pp^- \oplus \Lambda_\pp^+$.
As a result, the local Levelt filtration of $\cal{E}$ at $\pp$ is $\cal{E}_\pp^\bullet = \big( \Lambda_\pp^- \subset \cal{E}^\pp \big)$.

Let $U_\II$ be a Stokes region with $\II = \set{i,i'}$ and with polar vertices $\pp, \pp'$ such that the spectral regions $U_i, U_{i'}$ are respectively incident to the preimages $\pp_-, \pp'_-$.
By construction, if $\smash{\cal{L}_{i^{(\prime)}} \coleq \evat{\cal{L}}{U_{i^{(\prime)}}}}$ and $\smash{\Lambda_{i^{(\prime)}} \coleq \pi_\ast \cal{L}_{i^{(\prime)}}}$, then $\cal{E}_\II = \Lambda_i \oplus \Lambda_{i'}$.
Of course, $\cal{L}_{i}$ is the unique continuation of $\cal{L}_{\pp}^-$ from $U_{\pp}^-$ to $U_i$, and therefore $\Lambda_i$ is the unique continuation of $\Lambda_\pp^-$ from $U_\pp$ to $U_\II$.
Same for $\Lambda_{i'}$.
As a result, the direct sum $\Lambda_i \oplus \Lambda_{i'}$ is nothing but the transverse intersection $\cal{E}_{\pp,\II}^\bullet \pitchfork \cal{E}_{\pp',\II}^\bullet$ of Levelt filtrations $\cal{E}_\pp^\bullet, \cal{E}_{\pp'}^\bullet$ continued to $U_\II$.
This demonstrates the fact that $\nabla$ is transverse with respect to $\Gamma$, so $(\cal{E}, \nabla, \MM) \in \Conn_X^2 (\Gamma)$.

\begin{prop}{181109131348}
The correspondence $(\cal{L}, \de, \mu) \mapsto (\cal{E}, \nabla, \MM)$ extends to a functor
\eqntag{
	\pi_\ab^\Gamma : \Conn_\sf{\Sigma}^1 \too \Conn_X^2 (\Gamma)
\fullstop
}
\end{prop}

This follows immediately from the commutative square \eqref{181103130752}.
We call $\pi_\ab^\Gamma$ the \dfn{nonabelianisation functor}, and the image $(\cal{E}, \nabla, \MM)$ of $(\cal{L}, \de, \mu)$ under $\pi_\ab^\Gamma$ the \dfn{nonabelianisation} of $(\cal{L}, \de, \mu)$ with respect to the Stokes graph $\Gamma$.
Finally, our \hyperlink{191115100309}{Main Theorem \ref*{191115100309}} follows from the following proposition.

\begin{prop}{180508230111}
The functors $\pi^\ab_\Gamma, \pi_\ab^\Gamma$ form a pair of inverse equivalences of categories.
\end{prop}

\begin{proof}
Given $(\cal{E}, \nabla, \MM) \in \Conn_X^2 (\Gamma)$, let $(\cal{L}, \de, \mu) \in \Conn_\sf{\Sigma}^1$ be its image under $\pi^\ab_\Gamma$.
By construction, the Voros cocycle $\Voros$ applied to $\cal{L}$ is the cocycle $\VV$ from \eqref{181126185446}.
\Autoref{181109154245} gives a canonical isomorphism $\pi_\ab^\Gamma \pi^\ab_\Gamma \cal{E} \iso \cal{E}$, so $\pi_\ab^\Gamma \pi^\ab_\Gamma \Rightarrow \Id$.

The converse is clear from the discussion above of diagonal decompositions and transversality (\autoref{191116103716}), so we will be brief.
Given $(\cal{L}, \de, \mu) \in \Conn_\sf{\Sigma}^1$, let $(\cal{E}, \nabla, \MM) \in \Conn_X^2 (\Gamma)$ be its nonabelianisation, and suppose $\cal{L}'$ is the abelianisation of $\cal{E}$.
First, we have $\cal{L}_\pp^\pm \iso \cal{L}_\pp^{\prime \pm}$ for every $\pp \in D$.
If $U_{i} \subset \sf{\Sigma}$ is a spectral region with sink polar vertex $\pp_-$, then $\Lambda_i = \pi_\ast \cal{L}_i$ is the unique continuation of $\Lambda_\pp^-$.
Both $\cal{L}_i$ and $\cal{L}'_i$ are the unique continuations of $(\pi_\pp^-)^\ast \Lambda_\pp^-$ to $U_i$, we get $\cal{L}'_i \iso \cal{L}_i$.
Thus, $\cal{L}, \cal{L}'$ are canonically isomorphic over $\sf{\Sigma} \setminus R$, and because their extensions over $R$ are unique, this isomorphism also extends over $\sf{\Sigma}$.
So $\cal{L} \iso \cal{L}' = \pi^\ab_\Gamma \pi_\ab^\Gamma \cal{L}$, and hence $\id \Rightarrow \pi^\ab_\Gamma \pi_\ab^\Gamma$.
\end{proof}

{\footnotesize
\bibliographystyle{nikolaev}
\bibliography{/Users/Nikita/Documents/Library/References}
}

\end{document}